\documentclass{amsart}

\usepackage{comment}
\usepackage{amsfonts}
\usepackage{amssymb}
\usepackage{amsthm}
\usepackage{mathtools}   
\usepackage[alphabetic]{amsrefs}
\usepackage{csquotes}
\usepackage{footmisc}
\usepackage{cite}
\usepackage{xpatch}
\xpatchcmd{\bibsection}{*}{}{}{} 
\usepackage{longtable}
\usepackage[all,cmtip]{xy}
\usepackage{fancyhdr}
\usepackage[normalem]{ulem}

\usepackage{graphicx}
\usepackage{hyperref}

\usepackage{tikz-cd}
\usepackage{tikz}
\usepackage{pifont}

\usepackage{aurical}
\usepackage{stmaryrd}
\usepackage[T1]{fontenc}
\usepackage[outline]{contour}
\usepackage{xcolor}
\usepackage{etoolbox}
\usepackage{array}
\usepackage{lipsum}

\usepackage{enumitem}

\usepackage{parskip}

\setlist[enumerate,1]{label={(\alph*)}}

\makeatletter
\DeclareFontEncoding{LS2}{}{\@noaccents}
\makeatother
\DeclareFontSubstitution{LS2}{stix}{m}{n}
\DeclareSymbolFont{largesymbolsstix}{LS2}{stixex}{m}{n}
\DeclareMathDelimiter{\lbrbrak}{\mathopen}{largesymbolsstix}{"EE}{largesymbolsstix}{"14}
\DeclareMathDelimiter{\rbrbrak}{\mathclose}{largesymbolsstix}{"EF}{largesymbolsstix}{"15}

\newcommand{\git}{\mathbin{
  \mathchoice{\mkern-3mu/\mkern-6mu/\mkern-3mu}
    {\mkern-3mu/\mkern-6mu/\mkern-3mu}
    {/\mkern-5mu/}
    {/\mkern-5mu/}}}

\newcommand\blfootnote[1]{%
  \begingroup
  \renewcommand\thefootnote{}\footnote{#1}%
  \addtocounter{footnote}{-1}%
  \endgroup
}

\newcommand{\genlegendre}[4]{%
  \genfrac{(}{)}{}{#1}{#3}{#4}%
  \if\relax\detokenize{#2}\relax\else_{\!#2}\fi
}

\newcommand{\WIP}[1]{\scalebox{#1}{\begin{tikzpicture}[limb/.style={line cap=round,line width=1.5mm,line join=bevel}]
\draw[line width=2mm,rounded corners,fill=yellow] (-2,0) -- (0,-2) -- (2,0) -- (0,2) -- cycle;
\fill (1.5mm,7mm) circle (1.5mm);
\fill(0,-7.5mm) -- ++(10mm,0mm) -- ++(120:2mm)--++(100:1mm)--++(150:2mm) arc (70:170:2.5mm and 1mm);
\draw[limb] (-7.5mm,-6.5mm)--++(70:4mm)--++(85:4mm) coordinate(a)--++(-45:5mm)--(-2.5mm,-6.5mm);
\fill[rotate around={45:(a)}] ([shift={(-0.5mm,0.55mm)}]a) --++(0mm,-3mm)--++
        (7mm,-0.5mm)coordinate(b)--++(0mm,4mm)coordinate(c)--cycle;
\draw[limb] ([shift={(-0.6mm,-0.4mm)}]b) --++(-120:5mm) ([shift={(-0.5mm,-0.5mm)}]c) --++
        (-3mm,0mm)--++(-100:3mm)coordinate (d);
\draw[ultra thick] (d) -- ++(-45:1.25cm);
\end{tikzpicture}}}

\robustify{\WIP}

\definecolor{cec1d24}{RGB}{236,29,36}
\definecolor{cffffff}{RGB}{255,255,255}

\let\int\relax\DeclareMathOperator*{\int}{int}

\newcommand{\what}[1]{{\widehat{#1}}}

\newcommand{\inn}{\operatorname{inn}}

\newcommand{\lra}{\longrightarrow}
\newcommand{\ra}{\rightarrow}

\newcommand{\bA}{\mathbb{A}}

\newcommand{\bC}{\mathbb{C}}

\newcommand{\bF}{\mathbb{F}}
\newcommand{\bG}{\mathbb{G}}

\newcommand{\bP}{\mathbb{P}}
\newcommand{\bQ}{\mathbb{Q}}
\newcommand{\bR}{\mathbb{R}}

\newcommand{\bZ}{\mathbb{Z}}

\newcommand{\Qbar}{{\overline{\mathbb{Q}}}}

\newcommand{\cC}{\mathcal{C}}

\newcommand{\cG}{\mathcal{G}}
\newcommand{\cH}{\mathcal{H}}

\newcommand{\cN}{\mathcal{N}}
\newcommand{\cO}{\mathcal{O}}

\newcommand{\cS}{\mathcal{S}}

\newcommand{\la}[1]{\,^{#1}\!} 

\newcommand{\mf}[1]{\mathfrak{#1}} 

\newcommand{\fm}{\mathfrak{m}}

\newcommand{\fS}{\mathfrak{S}}

\newcommand{\Aut}{\operatorname{Aut}}

\newcommand{\Inn}{\operatorname{Inn}}
\newcommand{\Out}{\operatorname{Out}}

\newcommand {\spmatrix}[4]{\left[\begin{smallmatrix}#1&#2\\#3&#4\end{smallmatrix}\right]}

\newcommand{\thmatrix}[9]{\left[\begin{array}{rrr}#1&#2&#3\\#4&#5&#6\\#7&#8&#9\end{array}\right]}
\newcommand{\diag}{{\operatorname{diag}}}

\newcommand{\ctvector}[3]{\left[\begin{array}{c}#1\\#2\\#3\end{array}\right]}

\newcommand{\Ad}{\operatorname{Ad}}

\newcommand{\rad}{\operatorname{rad}}

\newcommand{\Gal}{{\operatorname{Gal}}}
\newcommand{\id}{\operatorname{id}}
\newcommand{\Hom}{\operatorname{Hom}}

\newcommand{\Conf}{\operatorname{Conf}}
\newcommand{\UConf}{\operatorname{UConf}}

\newcommand{\Epi}{\operatorname{Epi}}

\newcommand{\GL}{\operatorname{GL}}
\newcommand{\SL}{\operatorname{SL}}
\newcommand{\PSL}{\operatorname{PSL}}

\newcommand{\PSU}{\operatorname{PSU}}
\newcommand{\PU}{\operatorname{PU}}
\newcommand{\SU}{\operatorname{SU}}
\newcommand{\J}{\operatorname{J}}
\newcommand{\U}{\operatorname{U}}

\newcommand{\Or}{\operatorname{O}}
\newcommand{\SO}{\operatorname{SO}}
\newcommand{\Sz}{\operatorname{Sz}}

\newcommand{\Stab}{\operatorname{Stab}}

\newcommand{\Tr}{\operatorname{Tr}}

\newcommand{\Res}{\operatorname{Res}}
\newcommand{\ab}{{\operatorname{ab}}}

\newcommand{\tr}{\operatorname{tr}}

\newcommand{\rightiso}{\stackrel{\sim}{\longrightarrow}}

\newcommand{\ol}[1]{{\overline{#1}}}

\newcommand{\Spec}{\operatorname{Spec}}

\newcommand{\B}{\operatorname{B}}

\newcommand{\zar}{\text{zar}}

\newcommand{\Rp}{\operatorname{Re}}

\newcommand{\gap}{\vspace{0.4cm}}

\theoremstyle{definition}\newtheorem{defn}{Definition}[section]
\theoremstyle{remark}
\theoremstyle{remark}
\theoremstyle{remark}\newtheorem{remark}[defn]{Remark}
\theoremstyle{plain}\newtheorem{question}[defn]{Question}
\theoremstyle{remark}\newtheorem*{remark*}{Remark}
\theoremstyle{remark}
\theoremstyle{remark}
\theoremstyle{definition}
\theoremstyle{remark}
\theoremstyle{definition}

\theoremstyle{plain}
\theoremstyle{plain}
\theoremstyle{plain}\newtheorem{prop}[defn]{Proposition}
\theoremstyle{plain}\newtheorem{thm}[defn]{Theorem}
\theoremstyle{plain}
\theoremstyle{plain}\newtheorem{lemma}[defn]{Lemma}
\theoremstyle{plain}\newtheorem{cor}[defn]{Corollary}
\theoremstyle{plain}\newtheorem{conj}[defn]{Conjecture}
\theoremstyle{plain}\newtheorem*{thm*}{Theorem}
\theoremstyle{plain}\newtheorem*{conj*}{Conjecture}
\theoremstyle{plain}\newtheorem*{prop*}{Proposition}

\theoremstyle{plain}

\newcommand{\note}[1]{{\color{blue} \sf $\spadesuit$ NOTE: [#1]}}

\title{Finite simple characteristic quotients of the free group of rank 2}

\renewcommand{\int}{\operatorname{int}}
\author{William Y. Chen, Alexander Lubotzky, Pham Huu Tiep}

\begin{document}

\begin{abstract}
In this paper we describe how to explicitly construct infinitely many finite simple groups as characteristic quotients of the rank 2 free group $F_2$.	 This shows that a ``baby'' version of the Wiegold conjecture \cite{Lub11} fails for $F_2$, and provides counterexamples to two conjectures in the theory of noncongruence subgroups of $\SL_2(\bZ)$ \cite{Chen18}. Our main result explicitly produces, for every prime power $q\ge 7$, the groups $\SL_3(\bF_q)$ and $\SU_3(\bF_q)$ as characteristic quotients of $F_2$. Our strategy is to study specializations of the Burau representation for the braid group $B_4$, exploiting an exceptional relationship between $F_2$ and $B_4$ first observed by Dyer, Formanek, and Grossman \cite{DFG82}. Weisfeiler's strong approximation theorem guarantees that our specializations are surjective for infinitely many primes, but it is not effective. To make our result effective, we give another proof of surjectivity via a careful analysis of the maximal subgroup structures of $\SL_3(\bF_q)$ and $\SU_3(\bF_q)$. These examples are minimal in the sense that no finite simple group of the form $\PSL_2(\bF_q)$ appears as a characteristic quotient of $F_2$.
\end{abstract}

\maketitle

\blfootnote{A. L. was supported by the  European Research Council (ERC) under the European Union’s Horizon 2020 (grant agreement No 882751).}
\blfootnote{P. H. T. gratefully acknowledges the support of the NSF (grants DMS-1840702 and DMS-2200850), the Simons Foundation, and the Joshua Barlaz Chair in Mathematics.}

\section{Introduction}
Let $F_n$ be the free group on $n$ generators, $G$ a finite group, and $\Epi(F_n,G)$ the set of epimorphisms from $F_n$ onto $G$. The group $\Aut(G)\times\Aut(F_n)$ acts on $\Epi(F_n,G)$ as follows: for $\alpha\in\Aut(F_n),\beta\in\Aut(G)$, and $\varphi\in\Epi(F_n,G)$,
$$(\beta,\alpha).\varphi := \beta\circ\varphi\circ\alpha^{-1}$$
A long standing conjecture attributed to Wiegold \cite[Conjecture 2.5.4]{Pak00} asserts that if $n\ge 3$ and $G$ is a finite simple group, then this action is transitive. As $\Epi(F_n,G)/\Aut(G)$ can be identified with the set
$$\cN_n(G) := \{N\lhd G\;|\; F_n/N\cong G\}$$
of normal subgroups of $F_n$ with quotient isomorphic to $G$, this conjecture can be reformulated as:
\begin{conj}[Wiegold] If $n\ge 3$ and $G$ is a finite simple group, then $\Aut(F_n)$ acts transitively on $\cN_n(G)$.
\end{conj}
In the language of group presentations, the conjecture roughly says that for any $n\ge 3$, any finite simple group $G$ has a unique presentation as an $n$-generated group. It has been known for a long time that the statement of the conjecture is false for $n = 2$. In fact, in 1951 B. H. Neumann and H. Neumann noticed that already for $G = A_5$, $\Aut(F_2)$ acts with two orbits on $\cN_2(A_5)$ \cite{NN51}. A more quantitative statement was proved by Garion and Shalev, who showed that $|\cN_2(G)/\Aut(F_2)|\ra\infty$ for any infinite family of nonabelian finite simple groups $G$. I.e., the number of orbits goes to infinity as the order of the group goes to infinity.\footnote{The case where $G$ ranges over groups of type $\PSL_2(\bF_q)$ was already known to Evans \cite{Evans85}.}

In \cite{Lub11}, the second author suggested a ``baby'' version of the Wiegold conjecture, which seems to be of independent interest:

\begin{conj}[``Baby Wiegold'', Lubotzky \cite{Lub11}] For $n\ge 3$ and $G$ a finite simple group, $\Aut(F_n)$ has no fixed points on $\cN_n(G)$.	
\end{conj}

It is easy to see that the ``Baby Wiegold conjecture'' is equivalent to
\begin{conj} For $n\ge 3$, $F_n$ does not have a characteristic finite index subgroup $R$ for which $F_n/R$ is a finite simple group.	
\end{conj}

The goal of this note is to show that even the baby Wiegold conjecture fails for $n = 2$:
\begin{thm}[see \S\ref{ss_burau_roots}]\label{thm_initial} The free group on two generators has infinitely many characteristic subgroups $N$ with $F_2/N$ being a finite simple group. In particular, the Baby Wiegold conjecture fails for $n = 2$.
\end{thm}

The strong failure of Wiegold's conjecture in the case $n = 2$ as shown in \cite{Evans85,GS09} can be explained as follows: It is a classical result of Nielsen that if $a,b$ and $a',b'$ are two bases for $F_2$, then the commutator $[a,b]$ is conjugate to $[a',b']^{\pm1}$. It follows that for $\varphi\in\Epi(F_2,G)$, the union of the conjugacy classes of $[\varphi(a),\varphi(b)]^{\pm1}$ is an invariant of the $\Aut(F_2)$-action. Since almost every pair of elements of a finite simple group $G$ generates, it is natural to expect that the number of conjugacy classes of commutators of generating pairs goes to $\infty$ as $|G|\ra\infty$; a main achievement of \cite{GS09} is making this idea precise. However, their method of proof does not give any information on fixed points, and so does not say anything about baby Wiegold.

In this paper we take a very different route. In \cite{Lub11}, a connection was established between Wiegold's conjecture and the following well-known conjecture.

\begin{conj}\label{conj_vs} For every finite dimensional linear representation $\rho$ of $\Aut(F_n)$, $n\ge 3$, $\rho(\Inn(F_n))$ is virtually solvable.
\end{conj}

In fact, it was shown in \cite{Lub11} that Wiegold implies Conjecture \ref{conj_vs} when the Zariski closure of $\rho(\Aut(F_n))$ is connected (also see the more recent manuscript of Gelander \cite{Gel22}). Conjecture \ref{conj_vs} fails for $\Aut(F_2)$; in fact $\Aut(F_2)$ is known to have a faithful linear representation. Indeed, in \cite{DFG82}, Dyer, Formanek and Grossman showed that $\Aut(F_2)$ has a faithful representation if and only if the braid group $B_4$ does. At the time the problem for $B_4$ was an open problem, but later it was shown in \cite{Big01, Kram00, Kram02} that all braid groups are linear, via the so-called ``Lawrence-Krammer representation''. Consequently, $\Aut(F_2)$ is also linear.


In this paper we will use the related \emph{Burau representation} $\rho : B_4\ra\GL_3(\bZ[q,q^{-1}])$ to produce infinitely many finite simple groups as characteristic quotients of $F_2$, thereby proving Theorem \ref{thm_initial} and showing that baby Wiegold fails for $F_2$. A key tool is the strong approximation theorem of Weisfeiler and Pink, which allows us to deduce that the Burau representation is surjective onto the $k_0$-valued points of a certain algebraic group, where $k_0$ is an appropriate finite field quotient of the ring $\bZ[q+q^{-1}]$. A similar strategy was employed by Funar and Lochak \cite{FL18} to show that for every $g\ge 2$, the fundamental group $\Pi_g$ of a closed surface of genus $g$ has infinitely many characteristic finite simple quotients. For them, the Burau representation is replaced by quantum representations of mapping class groups.

The strategy using strong approximation is quite general, and tells us which types of finite simple groups we can expect to produce as characteristic quotients of $F_2$. However it requires us to discard an indeterminate finite set of primes and hence the method does not specify even a single group as a characteristic quotient of $F_2$. Because of this, we were motivated to prove the following more precise result:

\begin{thm}[see \S\ref{section_tiep}]\label{thm_main} Let $q$ be any prime power $\ge 7$. Then the finite simple groups $\PSL_3(\bF_{q})$ and $\PSU_3(\bF_{q})$ are both characteristic quotients of $F_2$. The finite simple groups $\PSU_3(\bF_4)$ and $\PSU_3(\bF_5)$ are also characteristic quotients of $F_2$. The cases $\PSL_3(\bF_q)$ $(q = 2,3,4,5)$ and $\PSU_3(\bF_q)$ $(q = 2,3)$ are not characteristic quotients of $F_2$.
\end{thm}

Here, $\PSU_3(\bF_{q})$ denotes the projective special unitary group of a 3-dimensional Hermitian space over $\bF_{q^2}$ relative to the $\bF_{q}$-linear involution $a\mapsto a^{q}$. In the proof of Theorem \ref{thm_main}, the surjectivity provided by strong approximation is replaced by a careful analysis of the maximal subgroup structure\footnote{This was worked out by Mitchell and Hartley in the early 1900's and does not use the classification of finite simple groups.} of $\SL_3(\bF_q)$ and $\SU_3(\bF_q)$; moreover, these characteristic quotients are constructed using \emph{explicit} specializations of the Burau representation. We have checked using GAP \cite{GAP4} that $\PSU_3(\bF_4)$ is the smallest finite simple characteristic quotient of $F_2$.

An overview of the strong approximation strategy is given in Proposition \ref{prop_main}, where we give general conditions under which a representation of $\Aut(F_2)$ would give infinitely many concrete counterexamples to baby Wiegold for $n = 2$. We expect that there are many representations satisfying such conditions, though in this paper we only consider those coming from various specializations of the Burau representation of $B_4$. Other candidates include the Jones representation \cite{Jones87}, and the Lawrence-Krammer representation (see \cite{Big03} and \cite[\S3]{Sto10}).\footnote{In \cite{Big03} it is noted that the description of the Lawrence-Krammer representation in the earlier paper \cite{Big01} has a sign error.} In particular, the counterexamples $G$ which we get are all of bounded Lie rank. It would be interesting to find examples of unbounded rank. Another question of interest is whether there exist an alternating group $A_n$ which appears as a characteristic quotient of $F_2$. Our method cannot answer this last problem.

We will also show that $\PSL_2(\bF_q)$ is never a characteristic quotient of $F_2$ (Corollary \ref{cor_SL2}), so in some sense the quotients $\PSL_3(\bF_q),\PSU_3(\bF_q)$ that we find are the simplest finite simple groups of Lie type which appear as characteristic quotients of $F_2$. Our strategy for this is to show that the fixed points of the $\Out(F_2)$-action on the $\bF_q$-points of the character variety for $\SL_2$-representations of $F_2$ correspond to nonsurjective representations. This approach is made possible by the known integral structure of this character variety. A similar understanding for the $\SL_3$ and $\SU_3$-character varieties would in theory allow for a similar approach to Theorem \ref{thm_main}, though the equations would be notably more complicated \cite{Law07, Nak02, Sik01}, and it would not shed any light on where these fixed points come from.

Finally, our main theorem provides counterexamples to two conjectures in the theory of noncongruence modular curves \cite[Conjectures 4.4.1 and 4.4.15]{Chen18}. Let $\Aut^+(F_2)$ be the subgroup inducing determinant 1 automorphisms of the abelianization $\bZ^2$. The conjectures would assert that for nonsolvable groups $G$, the $\Aut^+(F_2)$-stabilizers of any surjection $\varphi : F_2\ra G$ should map to a \emph{noncongruence} subgroup of $\SL_2(\bZ)$. In \S\ref{section_noncongruence} we explain why this is a reasonable conjecture, and why our Theorem \ref{thm_main} provides counterexamples.

\subsection*{Acknowledgements}
This paper originated from discussions held at the Institute for Advanced Study in Princeton during the Fall of 2021. We thank the IAS for its hospitality. We also thank Gabriel Navarro for some helpful discussions, and Danny Neftin for his contributions during the early stages of the paper. We would also like to thank the referee for valuable remarks, which in particular inspired the results in \S\ref{section_markoff}.

\subsection*{Organization of the paper}
In \S\ref{section_burau}, we introduce the Burau representation of the braid group $B_4$ and the associated Hermitian form and unitary group, and show how it can be used to produce characteristic quotients of $F_2$.

In \S\ref{section_characteristic}, we describe Weisfeiler and Pink's strong approximation theorem, and show, following \cite{Lub11}, how it can be used to produce infinitely many finite linear groups as invariant quotients of a group $F$ from a single characteristic 0 representation of $\Aut(F)$. We apply this to the Burau representation to deduce Theorem \ref{thm_initial}.

In \S\ref{section_tiep}, we prove the stronger Theorem \ref{thm_main} by carefully choosing specializations of the Burau representation, exploting the known structure of the maximal subgroups of $\SL_3(\bF_q)$ and $\SU_3(\bF_q)$.

In \S\ref{section_markoff}, we will show that $\PSL_2(\bF_q)$ is never a characteristic quotient of $F_2$ (for any prime power $q$).

In \S\ref{section_noncongruence}, we describe how our main theorem provides counterexamples to two conjectures of \cite{Chen18} in the theory of noncongruence modular curves.

\section{Invariant quotients and the Burau representation}\label{section_burau}

\subsection{Generalities on invariant quotients}
Let $F$ be a normal subgroup of a group $A$. The conjugation action of $A$ on $F$ yields a homomorphism $h : A\ra\Aut(F)$. We say that a homomorphism of abstract groups $\varphi : F\lra G$ is \emph{$h(A)$-invariant} (or \emph{$A$-invariant}) if for every $\alpha\in h(A)\subset\Aut(F)$, $\varphi\circ\alpha = \sigma\circ \varphi$ for some $\sigma\in\Aut(G)$, and \emph{strictly $A$-invariant} if we can moreover take $\sigma\in\Inn(G)$. If $\varphi$ is surjective, then whenever $\sigma$ exists, it is uniquely determined by $\alpha$, in which case we will write $\sigma = \varphi_*(\alpha)$. For $A$-invariant quotients $\varphi$, the map
\begin{eqnarray*}
	\varphi_* : A & \lra & \Aut(G) \\
	\alpha & \mapsto & \sigma
\end{eqnarray*}
is a homomorphism.

If $h(A) = \Aut(F)$, then a homomorphism $\varphi : F\ra G$ is $A$-invariant if and only if $\ker(\varphi)$ is a characteristic subgroup of $F$; in this case we say that $\varphi$ is a characteristic homomorphism, or a characteristic quotient if $\varphi$ is moreover surjective. We will make use of the following criterion for producing $A$-invariant homomorphisms.
\begin{prop}\label{prop_IQ} With the above notation, if $G$ sits as a subgroup inside a group $\tilde{G}$ such that $\varphi : F\ra G$ extends to a map $\tilde{\varphi} : A\ra \tilde{G}$, then $\varphi$ is $A$-invariant. If $\varphi$ is surjective and $Z(G) = 1$, then the converse is also true. 
\end{prop}
\begin{proof} The forward direction is clear. For the converse, take $G\hookrightarrow \tilde{G}$ to be $G\cong\Inn(G)\hookrightarrow\Aut(G)$.	
\end{proof}

\subsection{The braid group $B_4$ and the free subgroup $F$}

Recall that the braid group $B_n$ on $n$ strands is the group generated by $\sigma_1,\ldots,\sigma_{n-1}$ subject to the relations
\begin{itemize}
\item $[\sigma_i,\sigma_j] = 1$	if $|i-j| \ge 2$, and
\item $\sigma_i\sigma_{i+1}\sigma_i = \sigma_{i+1}\sigma_i\sigma_{i+1}$ if $1\le i\le n-2$.
\end{itemize}
The center of $B_n$ is cyclic of infinite order generated by $(\sigma_1\sigma_2\cdots\sigma_{n-1})^n$, and has abelianization $\bZ$, given by the word length map $\ell : B_n\lra\bZ$ defined by $\ell(\sigma_i) = 1$ for each $i$.

Let $F_n$ denote the free group on the generators $x_1,\ldots,x_n$. There is a well known action of $B_n$ on $F_n$ given by $\sigma_i(x_1,\ldots,x_n) = (x_1,\ldots,x_{i-1},x_{i+1}, x_{i+1}^{-1}x_ix_{i+1}, x_{i+2},\ldots,x_n)$. 
In this paper we make use of a different representation, which seems to be exceptional to the case $n = 4$. From now on, we let $x,y,F$ denote the following objects inside $B_4$:
$$x := \sigma_1\sigma_3^{-1},\qquad y := \sigma_2\sigma_1\sigma_3^{-1}\sigma_2^{-1},\qquad F := \langle x,y\rangle$$
Then $x,y$ are free generators of $F$, and $F$ is a \emph{normal} subgroup of $B_4$ \cite[p406]{DFG82}. The action of $B_4$ on $F$ by conjugation defines a homomorphism
\begin{eqnarray*}
\xi : B_4 & \lra & \Aut(F) \\
 \sigma_1 & \mapsto & (x,y)\mapsto (x,yx^{-1}) \\
 \sigma_2 & \mapsto & (x,y)\mapsto (y,yx^{-1}y) \\
 \sigma_3 & \mapsto & (x,y)\mapsto (x, x^{-1}y)
\end{eqnarray*}
The image of $\xi$ is the group of automorphisms of $F$ which act with determinant 1 on the abelianization $F^\ab\cong\bZ^2$, denoted $\Aut^+(F)$, and the kernel is exactly the center $Z(B_4)$ \cite[p406]{DFG82}. Thus we have an exact sequence
\begin{equation}\label{eq_B4}
1\lra Z(B_4)\hookrightarrow B_4\stackrel{\xi}{\lra}\Aut^+(F)\lra 1	
\end{equation}

where we recall that $Z(B_4) = \langle(\sigma_1\sigma_2\sigma_3)^4\rangle\cong\bZ$. In terms of invariant quotients, we have

\begin{prop}\label{prop_Autp_invariant} Let $\varphi : B_4\ra G$ be a homomorphism of groups. Then the restriction $\varphi|_F : F\ra \varphi(F)$ is $\Aut^+(F)$-invariant.	
\end{prop}

\subsubsection{Remarks on canonical locally constant sheaves on surfaces}\label{sss_canonical} Geometrically, Proposition \ref{prop_Autp_invariant} implies that representations of $B_4$ correspond to \emph{canonical} geometric objects on punctured elliptic curves. We first explain this for covering spaces. Let $S_g$ be a closed surface of genus $g$, and let $S_{g,n}$ be an $n$-punctured $S_g$. For $g = 1$, the fundamental group $\pi_1(S_{1,1})$ is free of rank 2, and we will identify $\pi_1(S_{1,1})$ with $F$. By covering space theory, the map $\varphi|_F : F\ra \varphi(F)$ determines a $\varphi(F)$-Galois covering space $C\ra S_{1,1}$; here, $\varphi|_F$ is the \emph{monodromy representation} attached to $C\ra S_{1,1}$. If $S_{1,1}'$ is another punctured torus, then by choosing an orientation-preserving homeomorphism $f : S_{1,1}'\ra S_{1,1}$, we can pull back $C$ to obtain a $\varphi(F)$-Galois covering $f^*C$ of $S_{1,1}'$, and the isomorphism class of the pullback $f^*C$ depends only the homotopy class of $f$. The group of homotopy classes of orientation-preserving self-homeomorphisms of $S_{1,1}$ is the \emph{mapping class group} $\Gamma(S_{1,1})$. If the isomorphism class of $C$ is invariant under pulling back by elements of $\Gamma(S_{1,1})$, then the isomorphism class of $f^*C$ is \emph{independent} of the choice of homeomorphism $f$. At the level of the homomorphism $\varphi|_F$, this is equivalent to the condition that for any $\gamma\in\Gamma(S_{1,1})$, there exists a $\sigma\in\Aut(\varphi(F))$ such that
$$\psi\circ\gamma_*\sim \sigma\circ\psi$$
where $\gamma_*$ denotes the outer automorphism of $F\cong\pi_1(S_{1,1})$ induced by $\Gamma(S_{1,1})$, and $\sim$ denotes ``is conjugate by some element of $G$''. Since the outer action of $\Gamma(S_{1,1})$ on $F$ factors through $\Out^+(F)$, Proposition \ref{prop_Autp_invariant} implies that this condition is satisfied\footnote{In fact, the outer representation $\Gamma(E^\circ)\ra\Out^+(F)$ is an \emph{isomorphism} \cite[\S8.2.7]{FM11}.}. Thus the existence of the homomorphism $\varphi : B_4\ra G$ implies that every punctured torus (in particular, every punctured elliptic curve) has a \emph{canonical} $\varphi(F)$-Galois cover.

If $G\le\GL_n(k)$ is a linear group over some ring $k$, then $\varphi|_F$ can also be viewed as a $k$-local system of rank $n$, and the same argument implies that every elliptic curve has a \emph{canonical} $k$-local system of rank $n$. The relation between the two perspectives is that covering spaces and local systems are both examples of \emph{locally constant sheaves}\footnote{In general, for a connected and locally simply connected topological space $X$, locally constant sheaves on $X$ are determined by representations of the fundamental group $\rho : \pi_1(E^\circ)\ra\Aut(S)$ for some object $S$. If $S$ is a set, $\Aut(S)$ is a permutation group, and $\rho$ corresponds to a covering space (a locally constant sheaf of sets). If $S$ is a group, $\rho$ corresponds to an $S$-local system (a locally constant sheaf of groups). See \cite[\S2.5-2.6]{Sza09}.}

In genus $g\ge 2$, the mapping class group $\Gamma(S_{g,1})$ plays the role for $S_g$ that $B_4$ did for $S_{1,1}$. In particular, there is an isomorphism $\Gamma(S_{g,1})\cong\Aut^+(\pi_1(S_g))$ \cite[p235]{FM11}, and hence representations of $\Gamma(S_{g,1})$ correspond to canonical locally constant sheaves on $S_g$. This gives a geometric interpretation to a result of Funar and Lochak \cite[Theorem 1.4]{FL18}.

\subsection{The unitary group}\label{sss_unitary_group}

Let $R$ denote the ring $\bZ[q,q^{-1}]$, and let $R_0 := \bZ[q+q^{-1}]\subset R$.

\begin{defn}[The Hermitian form $h$ {\cite[\S3]{Ven14}}]\label{def_hermitian} The ring $R = \bZ[q,q^{-1}]$ admits an involution given by the unique $\bZ$-algebra automorphism sending $q\mapsto q^{-1}$. We will often denote this involution by $r\mapsto \ol{r}$ for $r\in R$; we will also sometimes call it \emph{conjugation}. The subring of fixed points under the involution is $R_0 = \bZ[q+q^{-1}]$. Since $R \cong R_0[t]/(t^2-(q+q^{-1})t + 1)$, $R$ is a free $R_0$-module of rank 2. Let $h$ denote the sesquilinear form antilinear in the second coordinate (relative to the involution) on the free $R$-module $R^3$ given by the matrix
$$H := \thmatrix{\frac{(q+1)^2}{q}}{-q-1}{0}{-q^{-1}-1}{\frac{(q+1)^2}{q}}{-q-1}{0}{-q^{-1}-1}{\frac{(q+1)^2}{q}}$$
Thus we have $h(v,w) = v^tH\ol{w}$ for all $v,w\in R^3$. Since $H^t = \ol{H}$, we see that $h$ is a Hermitian form.
\end{defn}

The matrix $H$ has determinant \cite[Lemma 13]{Ven14}
\begin{equation}\label{eq_det_of_H}
\det(H) = \left(\frac{q+1}{q}\right)^3\frac{q^4-1}{q-1}	
\end{equation}

The unitary group for $h$ is by definition the group of isometries of the Hermitian space $(R^3,h)$. As a group scheme, this group should be understood as a subgroup of the restriction of scalars $\Res_{R/R_0}\GL_{3,R}$:
\begin{defn}\label{defn_Uh} The unitary group $\U_h$ is the group scheme over $R_0$ whose $B$-points (for any $R_0$-algebra $B$) is
$$\U_h(B) = \{g\in \GL_3(R\otimes_{R_0} B)\;|\; h(gv,gw) = h(v,w)\text{ for all $v,w\in (R\otimes_{R_0} B)^3$}\}$$
Here, the involution on $R\otimes_{R_0}B$ is induced from $R/R_0$. Let $\SU_h$ be the subgroup scheme consisting of matrices with determinant 1. If $B$ is any $R_0$ algebra, let $\U_{h,B} := \U_h\times_{R_0}B$ denote the base change of $\U_h$ to $B$. For more details on unitary groups, see \cite[\S3]{Ven14}, \cite[\S2.3.3]{PR94}, and \cite[Exercise 4.4.5]{Con14}
\end{defn}

Recall that the standard unitary group $U(n)$ can be viewed as a real form of $\GL_n$. Here is the analogous result in our case:

\begin{prop}\label{prop_split_prime} Let $k$ be an $R$-algebra which is a field, such that the images of $q,q^{-1}$ are distinct in $k$. The base change of $\U_h$ to the $R_0$-algebra $k$ is isomorphic to $\GL_{3,k}$, and the base change of $\SU_h$ to $k$ is isomorphic to $\SL_{3,k}$. In particular, this applies to $k = \bQ(q)$.
\end{prop}

\begin{remark} The assumption that $q,q^{-1}$ are distinct in $k$ is equivalent to assuming that $q\ne\pm1\in k$. In particular, the result still holds if $q$ maps to a primitive 4th root of unity, even though $h$ is degenerate in this case.
\end{remark}

\begin{proof} 
For any $R_0$-algebra $B$, $\U_{h,R_0}(B)$ is the subgroup of matrices $g\in\GL_3(B\otimes_{R_0}R)$ satisfying $g^tH\ol{g} = H$, where the conjugation on $g$ comes from the conjugation on $B\otimes_{R_0}R$ which in turn is induced from that of $R$. Similarly, the elements of $\SU_{h,R_0}(B)$ are the matrices in $\SL_3(B\otimes_{R_0}R)$ satisfying the same relation. If $B$ is any $k$-algebra and $q_B$ denotes the image of $q$ in $B$, then by our assumption on $k$, we have $q_B\ne q_B^{-1}$. It follows that $B\otimes_{R_0} R\cong B\times B$, with conjugation given by $(x,y)\mapsto (y,x)$. Writing $g\in\GL_3(B\times B) = \GL_3(B)\times\GL_3(B)$ as $(g_1,g_2)$, the relation $g^tH\ol{g} = H$ becomes $(g_1^t,g_2^t)(H,\ol{H})(g_2,g_1) = (H, \ol{H})$. This implies $g_1^tHg_2 = H$, and $g_2^t\ol{H}g_1 = \ol{H}$, but since $\ol{H} = H^t$, these two relations are the same, so we are left with the single relation $g_1^tHg_2 = H$. Thus we can take $g_1$ to be anything in $\GL_3(B)$, and then $g_2$ is forced to be $H^{-1}g_1^{-t}H$.	 This shows that $\U_{h,k}(B) = \GL_3(B)$ and $\SU_{h,k}(B) = \SL_3(B)$ as desired.
\end{proof}

\subsection{The representation}\label{ss_the_representation}

The Burau representation arises from the monodromy action on the middle homology of a certain manifold, which is isomorphic to $R^3$ as abelian groups. The Hermitian form $h$ is constructed using the intersection form on this homology group, relative to which the Burau representation acts via isometries. An explicit description is given below. See \cite{Ven14}, \cite[\S2.2]{Sto10} for more details.

\begin{defn} The Burau representation for $B_4$ is the representation
$$\rho := \rho_{B,4} : B_4\lra\U_h(R_0)\subset \GL_3(R)$$
given by
$$(\sigma_1,\sigma_2,\sigma_3)\quad\mapsto\quad\left(\thmatrix{-q}{1}{0}{0}{1}{0}{0}{0}{1},\thmatrix{1}{0}{0}{q}{-q}{1}{0}{0}{1},\thmatrix{1}{0}{0}{0}{1}{0}{0}{q}{-q}\right)$$
\end{defn}
One can easily calculate the following:
$$\rho(x) = \thmatrix{-q}{1}{0}{0}{1}{0}{0}{1}{-q^{-1}},\quad \rho(y) = \thmatrix{1-q}{-q^{-1}}{q^{-1}}{1-q^2}{-q^{-1}}{0}{1}{-q^{-1}}{0}$$
We note that $\sigma_1,\sigma_2,\sigma_3$ are all conjugate in $B_4$, and hence $x = \sigma_1\sigma_3^{-1}$ and its conjugate $y = \sigma_2\sigma_1\sigma_3^{-1}\sigma_2^{-1}$ must have determinant 1 under any representation. In particular, the Burau representation restricts to a representation
$$\rho|_F : F\lra \SU_h(R_0)$$
If $k$ is any $\bZ[q,q^{-1}]$-algebra where $q+1$ is invertible, the vectors
$$v_1 := \ctvector{1}{q+1}{q},\quad v_2 := e_1 = \ctvector{1}{0}{0},\quad v_3 := e_3 = \ctvector{0}{0}{1}$$
give an eigenbasis of $k^3$ for $\rho(x)$ with corresponding eigenvalues $1,-q,-q^{-1}$ respectively. It follows that $\rho(y)$ has the same eigenvalues, with eigenbasis $\sigma_2v_1,\sigma_2v_2,\sigma_2v_3$. Since $\rho(x)$ has distinct eigenvalues over $\bQ(q)$, $\{v_1,v_2,v_3\}$ is orthogonal for $h$ in the sense that $h(v_i,v_j) = 0$ if $i\ne j$. Let $O := [v_1\; v_2\; v_3]$, and let $k$ be any $\bZ[q,q^{-1}]$-algebra where $q+1$ is invertible. Then $\{v_1,v_2,v_3\}$ is a basis of $k^3$, and relative to this basis, the Hermitian form $h$ on $k^3$ is given by the matrix
$$D := O^tH\ol{O} = \thmatrix{(q+q^{-1})^2 + 2(q+q^{-1})}{0}{0}{0}{\frac{(q+1)^2}{q}}{0}{0}{0}{\frac{(q+1)^2}{q}}$$

\subsection{Bootstrapping $\Aut^+(F)$-invariance to $\Aut(F)$-invariance}
For an $R_0$-algebra $k_0$, let $\rho_{k_0}$ denote the composition
$$\rho_{k_0} : B_4\stackrel{\rho}{\lra} U_h(R_0)\ra U_h(k_0)$$
where $\rho$ is the Burau representation. Via Proposition \ref{prop_Autp_invariant}, for every $R_0$-algebra $k_0$, $\rho_{k_0}$ gives rise to the $\Aut^+(F)$-invariant map
$$\rho_{k_0}|_F : F\lra \SU_h(k_0)$$
We will show that if $q+1$ is invertible in $k_0\otimes_{R_0}R$, then $\rho_{k_0}|_F$ is in fact characteristic, see Proposition \ref{prop_bootstrap} below. Since the unique nontrivial coset of $\Aut(F)/\Aut^+(F)$ consists of automorphisms of determinant $-1$, it suffices to show that for the automorphism $\alpha\in\Aut(F)$ given by
$$\alpha(x) = x^{-1},\quad\alpha(y) = y,$$
there exists an automorphism $\ol{\alpha}\in\Aut(\SU(k_0))$ such that
$$\rho_{k_0}|_F\circ\alpha = \ol{\alpha}\circ\rho_{k_0}|_F$$
If $X := \rho_{k_0}(x)$ and $Y := \rho_{k_0}(y)$, then we wish to find an automorphism $\ol{\alpha}$ of $\SU_h(k_0)$ sending $(X,Y)\mapsto (X^{-1},Y)$.

The easiest situation is if one can take $\ol{\alpha}$ to be inner, but it turns out this is impossible, because $\alpha$ fails to preserve traces\footnote{The shortest word $w\in F$ we could find for which $\tr\rho(w)\ne\tr\rho(\alpha(w))$ is $w = xyx^{-2}y^2$. In this case $\tr(\rho(xyx^{-2}y^2)) = q^5-2q^4+q^3+q^2-4q+4-2q^{-1}-q^{-2}+2q^{-3}-q^{-4}$ but
$\tr(\rho(x^{-1}yx^{2}y^2)) = -q^4+2q^3-q^2-2q+4-4q^{-1}+q^{-2}+q^{-3}-2q^{-4}+q^{-5}$.}; thus we are led to consider non-inner $\ol{\alpha}$. Over $\bC$, $\Out(\SU_h(\bC))\cong\bZ/2$, the nontrivial class being represented by ``complex conjugation with respect to a real structure''.\footnote{Since $\SU_h$ is simple and simply connected, its outer automorphisms come from automorphisms of Dynkin diagram. Since its Dynkin diagram is a line, we have $\Out(\SU_h)\cong\bZ/2$.} To describe the desired automorphism $\ol{\alpha}\in\SU_h(k_0)$, we begin by describing the arithmetic version of this conjugation.

Let $k := k_0\otimes_{R_0} R$, so $k$ is a free $k_0$-module of rank 2, and $k_0\subset k$ is the subalgebra fixed by the involution. We make the assumption that \emph{$q+1$ is invertible in $k$.} Then the vectors $\{v_1,v_2,v_3\}$ described in \S\ref{ss_the_representation} give an eigenbasis in $k^3$ for $X$. This is an orthogonal basis of $k^3$ relative to the Hermitian form $h$. Let $I := q-q^{-1}\in k$, and let $V_0$ be the $k_0$-span of $\{v_1,v_2,v_3\}$; this is a real structure in the sense that $h$ takes values in $k_0$ when restricted to $V_0$. Thus $k^3 = V_0\oplus IV_0$. Define $J$ to be the $k_0$-linear automorphism of $k^3$ given by
\begin{equation}\label{eq_J}
J := \id_{V_0}\oplus \left(-\id_{IV_0}\right)
\end{equation}
In particular, $Jv_i = v_i$ and $JIv_i = -Iv_i$ for all $1\le i\le 3$. Thus, $J$ is the analog of ``complex conjugation'' relative to the real structure $V_0$. Define $\sigma_J\in\Aut(U_h(k_0))$ by
$$\sigma_J(B) = JBJ^{-1} = JBJ$$
Then $\sigma_J$ represents a nontrivial class of $\Out(U_h(k_0))$.\footnote{More generally, every real structure $V_0$ on the Hermitian space $(k^3,h)$ determines, via this construction, an automorphism $\sigma_J$ of $\SU_h(k_0)$.} Since $v_1,v_2,v_3$ is an eigenbasis of $X$ with eigenvalues $1,-q,-q^{-1}$ respectively, we find that $v_1,v_2,v_3$ are eigenvectors of $JXJ$ with eigenvalues $1,-q^{-1},-q$, so $\sigma_J(X) = JXJ = X^{-1}$. It is not quite true that $\sigma_J(Y) = Y$, but it turns out this can be fixed by composing by the inner automorphism given by a suitable matrix centralizing $X$.

\begin{prop}\label{prop_bootstrap} Let $k_0$ be any $R_0$-algebra such that $q+1$ is invertible in $k := k_0\otimes_{R_0}R$. Let $J$ be the $k_0$-linear endomorphism of $k^3$ described in \eqref{eq_J}. Let $\delta$ be the matrix
$$\delta = \thmatrix{q^{-1}}{-q^{-2}}{0}{0}{-q^{-2}}{0}{0}{-q^{-2}}{q^{-3}}$$
Let $\ol{\alpha}$ be the automorphism of $\U_h(k_0)\subset\GL_3(k)$ given by
$$\ol{\alpha}(B) = \delta^{-1}JBJ\delta\quad\text{for $B\in \U_h(k_0)$}$$
Then $\ol{\alpha}(X) = X^{-1}$ and $\ol{\alpha}(Y) = Y$. In particular, $\rho_{k_0}|_F\circ\alpha = \ol{\alpha}\circ\rho_{k_0}|_F$, and $\rho_{k_0}|_F : F\ra\SU(k_0)$ is characteristic.
\end{prop}
\begin{proof} The key point is to find the appropriate matrix $\delta$. One can check that our $\delta$ centralizes $X$, and satisfies the advertised properties. We also note that
$$\delta v_1 = -q^{-2}v_1\qquad \delta v_2 = q^{-1}v_2\qquad \delta v_3 = q^{-3}v_3$$
so $\delta$ is invertible for any $R_0$-algebra $k_0$. The process of finding $\delta$ is a somewhat laborious calculation, which we omit.
\end{proof}

\begin{remark} The existence of the matrix $\delta$ can also be checked by showing that the traces of $\rho_{k_0}|_F(w)$ and $(\rho_{k_0}|_F\circ\alpha)(w)$ agree for sufficiently long words $w\in F$ \cite[Theorems 3.3, 3.4]{Pro76}, though this does not shed any light on \emph{why} it should exist. A better understanding of this would give guidance on when one can expect a general representation of $B_4$ to lead to characteristic quotients of $F$.



\end{remark}

\section{Characteristic quotients of $F$ via strong approximation}\label{section_characteristic}

The Burau representation gives an $\Aut^+(F)$-invariant map $\rho_{k_0}|_F : F\ra \SU_h(k_0)$ for any $R_0$-algebra $k_0$. In Proposition \ref{prop_bootstrap} we showed that this map is moreover characteristic. To obtain finite simple characteristic quotients of $F$, it remains to understand the image of $\rho_{k_0}|_F$. We will show that for suitable choices of finite fields $k_0$, the map is surjective onto $\SU_h(k_0)$. Taking the quotient by the center, we obtain simple characteristic quotients of $F$.

In this section we present a rather general method using the strong approximation theorems of Weisfeiler and Pink. For most number fields $K$, this establishes surjectivity for all but finitely many residue fields $k_0$ of $K$. This gives us infinitely many simple characteristic quotients, thus proving Theorem \ref{thm_initial}, but at the cost of a lack of specificity due to the mysterious finite collection of primes that must be excluded at every application of the strong approximation theorem.

In section \S\ref{section_tiep}, we describe a more direct method: By cleverly choosing specializations of $q$ to generators of various finite fields and exploiting the known structure of the maximal subgroups of $\SL_3(\bF_q)$ and $\SU_3(\bF_q)$, we are able to prove the more general Theorem \ref{thm_main}.

\subsection{Algebraic groups and strong approximation}

We recall some basic definitions of reductive group schemes. An algebraic group $G$ over a field $k$ is a smooth affine group scheme over $k$. It is reductive if $G_{\ol{k}}$ has trivial unipotent radical, semisimple if $G_{\ol{k}}$ has trivial radical, simple if $G_{\ol{k}}$ has no proper nontrivial connected normal subgroups, and simply connected if any central isogeny from a connected algebraic group to $G$ is trivial. When discussing reductive group schemes, we follow the conventions of Conrad \cite[Definition 3.1.1]{Con14}:

\begin{defn}[{Algebraic group schemes}] Let $S$ be a scheme. An algebraic group scheme over $S$ is a smooth group scheme affine over $S$ (sometimes we will say a ``smooth affine $S$-group scheme''). It is reductive if its geometric fibers are connected reductive algebraic groups in the usual sense. It is simple (resp. semisimple, simply connected) if its geometric fibers are simple (respectively semisimple, simply connected).
\end{defn}

If $v$ is a place of a number field $K$, let $K_v$ denote the completion of $K$ at $v$, and if $v$ is a finite place, let $\cO_v$ (resp. $k_v$) be the ring of integers (resp. residue field) of $K_v$. Recall that for a set of places $S$, $\bA_K^S$ denotes the ring of adeles away from $S$: the restricted product $\prod'_{v\notin S}K_v$. Given Proposition \ref{prop_IQ}, our general method for producing finite simple characteristic quotients is based on the following strong approximation theorem due to Weisfeiler and Pink:

\begin{thm}[Strong approximation for linear groups]\label{thm_SA} Let $K$ be a number field, let $G$ be a simple, simply connected linear algebraic group over $K$, and let $\Lambda\subset G(K)$ be a finitely generated Zariski-dense subgroup. Let $S$ denote a set of places $v$ of $K$ for which the image of $\Lambda$ in $G(\bA_K^S)$ is not discrete. Suppose $K$ is generated over $\bQ$ by the traces $\tr\Ad\gamma$. Then the closure of $\Lambda$ in $G(\bA_K^S)$ is open in $G(\bA_K^S)$.
\end{thm}

This theorem is a special case of the results of Weisfeiler \cite{Weis84} and Pink \cite[Theorem 0.3]{Pink00}.\footnote{Also see \cite[Window 9, Theorem 2]{LS03} for a discussion, and \cite{MVW84} for a somewhat weaker result.} The condition that $K$ is generated by the traces $\tr\Ad\gamma$ exists essentially to guarantee that $\Lambda$ is not a subgroup of $G(L)$ for some subfield $L\subset K$.\footnote{This ensures that ``Assumption 0.1'' of \cite{Pink00} is satisfied.}

Keeping the notation and assumptions as Theorem \ref{thm_SA}, let us further choose an embedding of $G$ into $\GL_{n,K}$. Relative to this embedding, define $G(\cO_v) := G(K)\cap\GL_n(\cO_v)$, and define $G(k_v)$ to be its image in $\GL_n(k_v)$. For varying $v$, these groups are related as follows: if $\cG$ denotes the schematic closure of $G$ inside $\GL_{n,\cO_K}$, then $\cG$ admits a unique structure of an affine flat group scheme over $\cO_K$ with generic fiber $G$ and $\cG(\cO_K) = G(K)\cap\GL_n(\cO_K)$ (see Proposition \ref{prop_schematic_closure} in the appendix). In particular for varying places $v$, the groups $G(\cO_v)$ (resp. $G(k_v)$) are all obtained as the $\cO_v$ (resp. $k_v$) valued points of the single group scheme $\cG$ over $\cO_K$. Moreover, there exists $b\in\cO_K$ such that the geometric fibers of $\cG_{\cO_K[1/b]}$ over $\cO_K[1/b]$ are connected reductive with identical root data.\footnote{More precisely, since the smooth locus is open and $G$ is smooth, $\cG$ is smooth over a dense open of $\Spec\cO_K$. Next, since the generic fiber is reductive, passing to a possibly smaller dense open, we may assume that the identity components of the fibers are all reductive \cite[Proposition 3.1.9]{Con14}, which further implies that all fibers are connected \cite[Proposition 3.1.12]{Con14}. Since $\cG$ \'{e}tale locally admits a split maximal torus, we see that the root datum of geometric fibers of $\cG$ are locally constant.} In particular, since $G$ is simple and simply connected, so is $\cG_{\cO_K[1/b]}$.

\begin{cor} For all but finitely many places $v$ of $K$, $\Lambda$ is contained in $G(\cO_v)$, and the reduction map $G(\cO_v)\ra G(k_v)$ maps $\Lambda$ surjectively onto $G(k_v)$.
\end{cor}
\begin{proof} Choose generators $\lambda_1,\ldots,\lambda_r$ for $\Lambda$. There is a finite set of places $S_1$, containing all infinite places, such that the images of $\lambda_i$ have $v$-integral coordinates in $\GL_n$ for every $v\notin S_1$. In particular $\Lambda$ is contained in the compact group 
$$G\Big(\prod_{v\notin S_1}\cO_v\Big)\subset G(\bA_K^{S_1}).$$
Here, this containment should be viewed as occuring inside $\GL_n(\bA_K^{S_1})$. Thus, the infinitude of $\Lambda$ implies that its image in $G(\bA_K^{S_1})$ is not discrete. Applying the strong approximation theorem with $S := S_1$, we deduce that $\Lambda$ is dense in an open subgroup of $G(\bA_K^{S_1})$, and hence there is a finite set of places $S_2$ such that $\Lambda$ is dense in $G(\cO_v)$ for all $v\notin S_1\cup S_2$. In particular, possibly excluding the places above $b$ (see preceding paragraph), Hensel's lemma implies that $\Lambda$ surjects onto the finite group $G(k_v)$ \cite[Theorem 3.5.62]{Poo17}.
\end{proof}

Combining Theorem \ref{thm_SA} and Proposition \ref{prop_IQ}, we arrive at a rather general method for producing invariant quotients:

\begin{prop}\label{prop_main} Let $A$ be a finitely generated group, and $F\unlhd A$ a finitely generated normal subgroup. Let $K$ be a number field, and let $G$ be a connected, simply connected, simple algebraic group over $K$ embedded inside an algebraic group $H$. Let
$$\rho : A\ra H(\ol{K})$$
be a homomorphism. Assume that
\begin{enumerate}
\item\label{cond_generation} $K$ is generated over $\bQ$ by the traces $\tr\Ad\gamma$ for $\gamma\in F$.
\item\label{cond_containment} $G\subset\ol{\rho(A)}^\zar$,
\item\label{cond_rationality} $\rho(F)\subset G(K)$, and
\item\label{cond_noncentrality} $\rho(F)\not\subset Z(G)(K)$.
\end{enumerate}
Then for all but finitely many places $v$ of $K$, $G(k_v)$, and hence the simple group $G(k_v)/Z(G(k_v))$ is an $A$-invariant quotient of $F$.
\end{prop}


\begin{proof} 
Since $G$ is simple, finite normal subgroups of $G$ are central, so the noncentrality hypothesis on $\rho(F)$ implies that it is infinite, and \ref{cond_containment} implies that $\ol{\rho(F)}^\zar$ is normal inside $G$. In particular, $\ol{\rho(F)}^\zar$ is a positive dimensional normal subgroup of $G$, and hence it is the entirety of $G$, so $\rho(F)$ is a finitely generated Zariski-dense subgroup of $G(K)$. Fixing an embedding $G\hookrightarrow\GL_{n,K}$, strong approximation implies that for all but finitely many places $v$ of $K$, $\rho(F)\subset G(\cO_v)$ and surjects onto $G(k_v)$, as desired. Possibly excluding finitely many additional primes, $\rho|_F : F\ra G(k_v)$ extends to a map $A\rightarrow \ol{\rho(A)}^\zar(k_v)$. It follows from \ref{prop_IQ} that $\rho|_F : F\ra G(k_v)$ is $A$-invariant, and so is its composition with the characteristic quotient $G(k_v)\ra G(k_v)/Z(G(k_v))$. Finally, we note for all but finitely many places $v$, $G(k_v)/Z(G(k_v))$ is simple \cite[Theorem 24.17]{MT11}.
\end{proof}

\begin{remark}\label{remark_FL} Let $S_g$ be a closed surface of genus $g\ge 2$, and let $S_{g,1}$ be a once-punctured $S_g$. Funar and Lochak's Theorem 1.4 \cite{FL18} can be understood as applying the above proposition to certain quantum representations of the mapping class group $\Gamma_{g,1}$ of $S_{g,1}$. Using the isomorphism $\Gamma_{g,1}\cong\Aut^+(\pi_1(S_g))$ \cite[p235]{FM11}, they obtain infinitely many finite simple characteristic quotients of $\pi_1(S_g)$ for every $g\ge 2$.
\end{remark}

\subsection{Specializations of the Burau representation at roots of unity}\label{ss_burau_roots}
For a nonzero complex number $t\in\bC^\times$, let $\rho_t, h_t, H_t,D_t$ denote the specializations of $\rho, h,H,D$ by simply setting $q = t$. If $t\in\bC^\times - \{\pm1\}$, we will write $\U_{h,t}$ for the base change $\U_{h,\bZ[t+t^{-1}]} = \U_h\times_{R_0}\bZ[t+t^{-1}]$, where $\bZ[t+t^{-1}]\subset\bC$ is given the $R_0$-algebra structure sending $q+q^{-1}\mapsto t+t^{-1}$. Similarly, $\SU_{h,t}$ is the subgroup of matrices of determinant 1.
When $t$ lies on the complex unit circle, the behavior of $h_t$ can be summarized as follows:
\begin{prop} If $t$ is a complex number with absolute value $|t| = 1$. Then $h_t$ is a Hermitian form on $\bC^3$ relative to complex conjugation. We view $\bR$ as a $\bZ[q+q^{-1}]$-algebra via the map $q+q^{-1}\mapsto t+t^{-1}$. Then we have
\begin{enumerate}[label=(\alph*)]
\item $h_t$ is nondegenerate if and only if $t\notin\{-1,i,-i\}$.
\item If $0 < \Rp t < 1$, then $h_t$ is positive definite, and $\U_{h,t}(\bR)\cong \U(3)$ is compact.
\item If $-1 < \Rp t < 0$ then $h_t$ is of indefinite type $(2,1)$, and $\SU_{h,t}(\bR)\cong \SU(2,1)$ is noncompact.
\end{enumerate}
\end{prop}
\begin{proof} The noncompactness of $\SU(2,1)$ can be seen directly from the observation that it contains every matrix of the form
$$\thmatrix{\sqrt{x^2+1}}{0}{x}{0}{1}{0}{x}{0}{\sqrt{x^2+1}}\qquad x\in\bR$$
and hence it is unbounded, hence noncompact. Everything else follows from the formulas for $\det(H)$ and $D$ given in \S\ref{ss_the_representation}.
\end{proof}

\begin{remark} The notation $U_{h,t}$ is intuitive, but can be misleading. For example, if $t = 1$, $h_t = h_1$ is a nondegenerate hermitian form on $\bC^3$, but $\U_{h,t}(\bR)$ is not the unitary group of $h_1$. The issue is that in this case $\U_{h,1}(\bR)$ is defined to be a subgroup of $3\times 3$ matrices with coefficients in the non-\'{e}tale algebra $\bR\otimes_{R_0}R\cong \bR[T]/(T-1)^2$. One can show that this is not even a reductive group (see \ref{prop_finite_groups}(c) below).
\end{remark}

\begin{remark} The element $t+t^{-1}$ may not be integral. In particular, after specialization, the group scheme $U_{h,t}$ may not surject onto $\Spec\bZ$.
\end{remark}

In \cite{Ven14}, Venkataramana studied the images of the specializations of the Burau representation (for various braid groups) when $q$ is a root of unity. Let $\zeta_d := e^{2\pi i/d}$. Then the groups $\U_{h,\zeta_d}, \SU_{h,\zeta_d},\PU_{h,\zeta_d}$ are defined over $\bZ[\zeta_d+\zeta_d^{-1}]$. In particular, for $d = 3,4,6$, these groups are defined over $\bZ$.

\begin{thm}[Venkataramana, McMullen, Deligne-Mostow]\label{thm_arithmetic} For $d = 3,4,6$, the image of $\rho_{\zeta_d}$ has finite index in $\U_{h,\zeta_d}(\bZ)$.
\end{thm}
\begin{proof} This is \cite[Theorem 3]{Ven14}, though in the case $d = 3$, this is due to McMullen \cite[Theorem 10.3]{Mcm13}, and Deligne-Mostow \cite{DM86}.
\end{proof}

For other values of $d$, the image of $\rho_{\zeta_d}$ can fail to be arithmetic see \cite[Corollary 11.8]{Mcm13} for the case $d = 18$. Note that $h_{\zeta_4}$ is degenerate and $\SU_{h,\zeta_6}(\bZ)$ is finite, since it is a lattice inside the compact group $\SU_{h,\zeta_6}(\bR)$. Thus for the purpose of constructing finite simple characteristic quotients, we are led to consider the specialization at $\zeta_3$.

\begin{prop}\label{prop_density} The Zariski closure of $\rho_{\zeta_3}(B_4)$ contains $\SU_{h,\zeta_3}$.
\end{prop}
\begin{proof} Since $U_{h,\zeta_3}(\bZ)\subset\GL_3(\bZ[\zeta_3])$, the image of the determinant is contained in the finite group $\bZ[\zeta_3]^\times$. It follows that $H := \rho_{\zeta_3}^{-1}(\SU_{h,\zeta_3}(\bZ))$ has finite index in $B_4$, and hence the image of $H$ is a lattice in $\SU_{h,\zeta_3}(\bZ)$. Since $\SU_{h,\zeta_3}(\bR)\cong \SU(2,1)$ is noncompact and simple, the Borel density theorem implies that $\rho_{\zeta_3}(H)$ is Zariski dense in $\SU_{h,\zeta_3}$ \cite[Corollary 4.5.6]{Mor15}. 
\end{proof}

Since $\rho(F)$ is not central in $\SU_{h,\zeta_3}$, applying strong approximation, we find:

\begin{thm}[Characteristic quotients of degree 1]\label{thm_degree_1} For all but finitely many primes $p$, $\SU_{h,\zeta_3}(\bF_p)$ is a characteristic quotient of $F_2$.	
\end{thm}
\begin{proof} The groups $\U_{h,\zeta_3},\SU_{h,\zeta_3}$ are group schemes over $\bZ[\zeta_3+\zeta_3^{-1}] = \bZ$. Applying Proposition \ref{prop_main}, to the Burau representation, taking $H = \U_{h,\zeta_3,\bQ}, G = \SU_{h,\zeta_3,\bQ}$, shows that $\SU_{h,\zeta_3}(\bF_p)$ is an $\Aut^+(F)$-invariant quotient. Proposition \ref{prop_bootstrap} shows that this is in fact a characteristic quotient.
\end{proof}

The finite group $\SU_{h,\zeta_3}(\bF_p)$ can be described in more familiar terms as follows. Recall that over a finite field $\bF_q$, up to isometry there is a unique Hermitian form on $(\bF_{q^2})^3$ relative to the extension $\bF_{q^2}/\bF_q$ \cite[Corollary 10.4]{Grove02}. Let $U_3(\bF_q)\subset\GL_3(\bF_{q^2})$ denote the corresponding group of isometries, and let $\SU_3(\bF_q)$ be the subgroup of determinant 1 matrices. The following proposition completes the proof of Theorem \ref{thm_initial}.

\begin{prop}\label{prop_finite_groups} For a prime $p$, the finite group $\SU_{h,\zeta_3}(\bF_p)$ is described as follows
	\begin{enumerate}
	\item If $p\equiv 1\mod 3$, then $\U_{h,\zeta_3}(\bF_p)\cong \GL_3(\bF_p)$, and $\SU_{h,\zeta_3}(\bF_p)\cong\SL_3(\bF_p)$
	\item If $p\equiv 2\mod 3$, then $\U_{h,\zeta_3}(\bF_p)\cong \U_3(\bF_{p})$, and $\SU_{h,\zeta_3}(\bF_p)\cong\SU_3(\bF_{p})\subset\SL_3(\bF_{p^2})$
	\item If $p = 3$, then $U_{h,\zeta_3}(\bF_3)$ is an extension of the orthogonal group $\Or_3(\bF_3)$ by the additive group $\bF_3^6$.
\end{enumerate}
\end{prop}
\begin{proof}

At primes $p\equiv 1\mod 3$, the statement follows from Proposition \ref{prop_split_prime} where $k = \bF_p$, with $R$-algebra structure given by sending $q$ to an element of order 3.

At primes $p\equiv 2\mod 3$, $\bZ[\zeta_3]\times_\bZ\bF_p\cong\bF_{p^2}$, so $U_{h,\zeta_3,\bF_p}$ is the usual unitary group relative to the extension $\bF_{p^2}/\bF_p$.

At $p = 3$, $\bZ[\zeta_3]\otimes_\bZ\bF_3 = \bF_3[x]/(x^2+x+1)$, where $x$ is the image of $\zeta_3\otimes 1$. Let $\epsilon := x-1$, then this becomes the ring of dual numbers $\bF_3[\epsilon]$ (where $\epsilon^2 = 0$), with conjugation given by $\epsilon\mapsto -\epsilon$. To simplify calculations we may use the orthogonal basis $\{v_1,v_2,v_3\}$ given in \S\ref{ss_the_representation} to replace $H$ with the diagonal matrix $D = D_{\zeta_3} = \diag(-1,1,1)$. For any field $k$, write $g\in\GL_3(k[\epsilon])$ as $g_0+\epsilon g_1$, where $g_0\in\GL_3(k)$ and $g_1\in M_3(k)$. The relation $g^tD\ol{g} = D$ becomes
$$g_0^tDg_0 + \epsilon(g_1^tDg_0 - g_0^tDg_1) = D$$
Thus we have $g_0^tDg_0 = D$, and $g_1^tDg_0 = g_0^tDg_1$. The first relation asserts that $g_0$ lies in the orthogonal group for $D$, and the second asserts that $g_0^tDg_1$ is symmetric. For fixed $g_0$, this is a linear equation in the coordinates of $g_1$, hence defines an affine space. For $g_0 = I_3$, this defines an additive group of dimension 6, so the map $\U_{h,\zeta_3,\bF_3}\lra \Or_{3,\bF_3}$ sending $g\mapsto g_0$ realizes $\U_{h,\zeta_3}(\bF_3)$ as an extension of $\Or_{3}(\bF_3)$ by $\bF_3^6$.
\end{proof}

\begin{remark} By the discussion in \S\ref{sss_canonical}, the Burau representation $\rho$ corresponds to a \emph{canonical} rank 3 $R$-local system on punctured elliptic curves with infinite monodromy (i.e., $\rho|_F$ has infinite image). This infinitude was critical in the proof of Theorem \ref{thm_degree_1}. We note that in general, low rank canonical local systems with infinite monodromy seem to be, in some sense, quite rare: in a recent preprint of Landesman and Litt \cite{LL23}, they show that if $r < \sqrt{g+1}$, then a canonical $\bC$-local system of rank $r$ on $S_{g,n}$ has \emph{finite monodromy} (i.e., the associated representation of $\pi_1(S_{g,n})$ has \emph{finite image} in $\GL_r(\bC)$).

\end{remark}

\subsection{Specialization at points on the unit circle which are not roots of unity}


In the previous section we showed that infinitely many groups of the form $\SU_3(\bF_p)$ and $\SL_3(\bF_p)$ are characteristic quotients of $F_2$. In \S\ref{section_tiep}, we will show the much stronger fact (Theorem \ref{thm_tiep}) that implies that for every $q\ge 7$, both $\SU_3(\bF_q)$ and $\SL_3(\bF_q)$ are finite simple characteristic quotients of $F_2$. However, because strong approximation forces us to discard an indeterminate finite set of primes, we cannot hope to prove such a result using strong approximation. Nonetheless, the method of strong approximation provides a general framework for understanding why something like Theorem \ref{thm_tiep} could be true. In this section we illustrate this with an intermediate result (Theorem \ref{thm_arbitrary_degree} below), obtained by considering more general specializations of the Burau representation.

\begin{thm}[{\cite[Theorem 1]{Sto10}}]\label{thm_sto} For any complex number $t$ which lies on the unit circle, is not a root of unity, and is sufficiently close to 1, $\rho_t(B_4)$ is dense in the Euclidean topology on $U_{h,t}(\bR)\subset\GL_3(\bC)$.
\end{thm}

Since $\rho_t(B_4)$ is dense in the Euclidean topology on $U_{h,t}(\bR)$, it is also Zariski dense. This leads to the following generalization of Theorem \ref{thm_degree_1}.

\begin{thm}[$\Aut^+(F)$-invariant quotients of arbitrary degree]\label{thm_arbitrary_degree} Let $d\ge 1$ be an integer. Then
\begin{enumerate}
\item For infinitely many primes $p$, $\SL_3(\bF_{p^d})$ is a characteristic quotient of $F$, and
\item For infinitely many primes $p$, $\SU_3(\bF_{p^d})$ is a characteristic quotient of $F$.
\end{enumerate}
\end{thm}

Unlike the specialization at $\zeta_3$, where we work over the explicit ring $\bZ[\zeta_3+\zeta_3^{-1}] = \bZ$, here, for every $p^d$, we must effectively describe an infinite sequence of algebraic numbers $t$ on the unit circle converging to $1$ for which the associated number fields have a prime of inertia degree $d$ over $p$. Moreover, there is also the added complexity of dealing with the trace field $\bQ(\tr\Ad\rho_t(F))$, which in the $t = \zeta_3$ case was just $\bQ$; ideally this trace field is just $\bQ(t+t^{-1})$, but we do not know if this is true. However it is not difficult to show that it contains $\bQ(t^3+t^{-3})$, and that for appropriately chosen $t$, we have $\bQ(t+t^{-1}) = \bQ(t^3+t^{-3})$. We begin with a few lemmas.

\begin{lemma}\label{lemma_trace_field} The trace field $\bQ(\tr\Ad\rho(\gamma)\;|\;\gamma\in F)$ contains $\bQ(q^3+q^{-3})$.	
\end{lemma}
\begin{proof} The trace of the adjoint action of $\gamma\in \SU_h(\bQ(q+q^{-1}))$ on the Lie algebra of $\SU_{h,\bQ(q+q^{-1})}$ can be calculated after base change to $\bQ(q)$, in which case it can be identified with the trace of the image of $\gamma$ in $\SU_h(\bQ(q))\cong\SL_3(\bQ(q))$ acting on the Lie algebra of $\SL_{3,\bQ(q)}$. This latter Lie algebra consists of the trace 0 matrices, and the adjoint action of $\gamma$ on a trace zero matrix $X$ is simply given by conjugation by $\gamma$. An explicit calculation shows that:
\begin{eqnarray*}
\tr\Ad\rho(x) = q^2-2q+ 2-2q^{-1}+q^{-2} \\
\tr\Ad\rho(x^2) = q^4+2q^2+2+2q^{-2}+q^{-4}
\end{eqnarray*}
From this, one calculates that
$$(\tr\Ad\rho(x))^2 - \tr\Ad\rho(x^2) - 6\tr\Ad\rho(x) = -4q^3 - 4q^{-3}$$
It follows that the trace field contains $\bQ(q^3+q^{-3})$.
\end{proof}

\begin{lemma}\label{lemma_nonzero_trace} Let $t\in\Qbar$ satisfy $|t| = 1$. Suppose $\bQ(t+t^{-1})$ is Galois over $\bQ$ and that $t+t^{-1}$ has nonzero trace over $\bQ$. Then $\bQ(t+t^{-1}) = \bQ(t^3+t^{-3})$.
\end{lemma}
\begin{proof} The result is clear if $t = \pm1$, so we may assume $t\ne\pm1$. Let $s_i := t^i + t^{-i}$. Let $d := [\bQ(s_1):\bQ]$. We have the relation $s_3 = s_1^3 - 3s_1$, or equivalently, $s_1$ is a root of the polynomial $f(X) := X^3 - 3X - s_3\in\bQ(s_3)[X]$. Since $\bQ(s_1)/\bQ(s_3)$ is Galois, either $f(X)$ is irreducible over $\bQ(s_3)$, or $f(X)$ is totally split. Moreover, it is irreducible if and only if $\deg_\bQ(s_3) = d/3$. We will show that the trace condition on $s_1$ implies that this is impossible, and hence $f(X)$ must be totally split, which is to say that $\bQ(s_1) = \bQ(s_3)$.

If $d$ is not divisible by 3, then $f(X)$ cannot be irreducible, so we may assume $d = 3m$ for some integer $m$. In this case $f(X)$ is irreducible if and only if $\deg_\bQ(s_3) = m$, in which case the minimal polynomial of $s_3$ gives a $\bQ$-linear dependence relation between $1,s_3,s_3^2,\ldots,s_3^m$. Using the relation $s_3 = s_1^3 - 3s_1$, we find that for any $k\ge 1$,
\begin{equation}\label{eq_s3k}
s_3^k = (s_1^3-3s_1)^k = s_1^{3k} + \binom{k}{1}(-3)s_1^{3k-2} + \binom{k}{2}(-3)^2s_1^{3k-4} + \cdots + \binom{k}{k}(-3)^ks_1^k	
\end{equation}

We will show that $1,s_3,\ldots,s_3^m$ are $\bQ$-linearly independent. Since $1,s_1,\ldots,s_1^{d-1}$ is a $\bQ$-basis of $\bQ(s_1)$, clearly $1,s_3,\ldots,s_3^{m-1}$ are $\bQ$-linearly independent. Using the minimal polynomial of $s_1$ to write $s_1^{3m} = s_1^d$ as a linear combination of $1,s_1,\ldots,s_1^{d-1}$, the coefficient of $s_1^{d-1}$ is equal to $\tr_{\bQ(s_1)/\bQ}(s_1)$, and hence nonzero. Using \eqref{eq_s3k}, this implies that the coefficient of $s_1^{d-1}$ in $s_3^m$ is nonzero, and hence $1,s_3,s_3^2,\ldots,s_3^m$ is $\bQ$-linearly independent. Thus $f(X)$ is totally split, and hence $\bQ(s_1) = \bQ(s_3)$.
\end{proof}

\begin{proof}[Proof of Theorem \ref{thm_arbitrary_degree}] For integers $n,k\ge 1$, let $\zeta$ be a primitive $n$th root of unity, and let
$$t_{\zeta,k} := \frac{\zeta+k}{\zeta^{-1}+k}$$
Clearly $t_{\zeta,k}\in\bQ(\zeta_n)$, $|t_{\zeta,k}| = 1$, $\arg(t_{\zeta,k}) = 2\arg(\zeta+k)$, and $t_{\zeta,k}\ra 1$ as $k\ra\infty$. Moreover, for all $k$ sufficiently large, we have
\begin{enumerate}
\item $t_{\zeta,k}$ generates the cyclotomic field $\bQ(\zeta_n)$

Indeed, the map $\zeta\mapsto t_{\zeta,k}$ is $\Gal(\Qbar/\bQ)$-equivariant, and for large $k$ the map $\zeta\mapsto t_{\zeta,k}$ is injective on primitive $n$th roots of unity. This shows that the Galois orbit of $t_{\zeta,k}$ has size $\phi(n) = [\bQ(\zeta_n):\bQ]$. 

\item $s_{\zeta,k} := t_{\zeta,k} + t_{\zeta,k}^{-1}$ generates the maximal totally real subfield of $\bQ(\zeta_n)$, and has nonzero trace over $\bQ$.

Indeed, $\tr_\bQ(s_{\zeta,k}) = 2\tr_\bQ t_{\zeta,k}$, which is nonzero since for $k > 2$, each Galois conjugate of $t_{\zeta,k}$ has positive real part. If $n \le 2$, the claim is trivial, so assume $n\ge 3$. Then as $\zeta$ runs over all $\frac{\phi(n)}{2}$ primitive $n$th roots of unity with positive imaginary part, for sufficiently\footnote{Let $S$ denote the set of primitive $n$th roots of 1 with positive imaginary part. Here we should choose $k$ so that every line passing through $\zeta+k,\zeta'+k$ for any pair $\zeta,\zeta'\in S$ does not pass through $0\in\bC$. This can always be done for $k$ large enough relative to $n$.} large $k$, the real parts of $t_{\zeta,k}$ are all distinct, and hence the elements $s_{\zeta,k}$ are all distinct. Thus $s_{\zeta,k}$ is a real number of degree $\ge \frac{\phi(n)}{2}$, and hence must generate the maximal totally real subfield of $\bQ(\zeta_n)$. 

\end{enumerate}
Thus for fixed $n$, let $\zeta_n := e^{2\pi i/n}$. Then choosing $k$ large enough, we may apply Stoimenow's theorem \ref{thm_sto} to deduce that the representation
\begin{eqnarray*}
\rho_{t_{\zeta_n,k}} : B_4 & \lra & U_{h,t_{\zeta,k}}(\bQ(s_{\zeta_n,k}))
\end{eqnarray*}
is Zariski dense. Since $\bQ(s_{\zeta_n,k})$ is Galois over $\bQ$ with nonzero trace, Lemmas \ref{lemma_trace_field}, \ref{lemma_nonzero_trace} imply that the trace field of $\rho_{t_{\zeta_n,k}}$ is equal to $\bQ(s_{\zeta_n,k})$. Propositions \ref{prop_main} and \ref{prop_bootstrap} then imply that for all but finitely many primes $\mf{p}$ of $\bQ(s_{\zeta_n,k})$ with residue field $\kappa$, the finite group $\SU_{h,t_{\zeta_n,k}}(\kappa)$ is a characteristic quotient of $F$.

To complete the proof, we want to show that we can choose $n$ (and $k\gg 0$) such that amongst the primes $\mf{p}$ of $\bQ(s_{\zeta_n,k})$ with residue field of degree $d$ over the prime field, infinitely many split (resp. remain inert). To do this, let $p > 3$ be a prime satisfying $p\equiv 1\mod 2d$, let $n := 3p$, and $k$ sufficiently large as above. The extension $\bQ(t_{\zeta_n,k})/\bQ(s_{\zeta_n,k})$ induces the following surjection of Galois groups
$$\Gal(\bQ(t_{\zeta_n,k})/\bQ)\cong(\bZ/n)^\times\lra (\bZ/n)^\times/\{\pm1\}\cong \Gal(\bQ(s_{\zeta_n,k})/\bQ)$$
Note $(\bZ/n)^\times\cong(\bZ/p)^\times\times(\bZ/3)^\times$. Let $x\in(\bZ/p)^\times$ be an element of order $2d$, and let $\alpha\in\Gal(\bQ(t_{\zeta_n,k})/\bQ)\cong(\bZ/n)^\times$ be the element corresponding to $(x,-1)\in (\bZ/p)^\times\times(\bZ/3)^\times$. Then $\alpha$ has order $2d$ but maps to an element of order $d$ in $\Gal(\bQ(s_{\zeta_n,k})/\bQ)\cong(\bZ/n)^\times/\{\pm1\}$. By Chebotarev density, infinitely many primes $\mf{p}$ of $\bQ(s_{\zeta_n,k})$ have Frobenius element $\alpha$; such primes have residue degree $d$ and remain inert in $\bQ(t_{\zeta_n,k})$. Using similar arguments as in Proposition \ref{prop_finite_groups}, for all but finitely many such primes, we have $\SU_{h,t_{\zeta_n,k}}(\kappa)\cong \SU_3(\bF_{p^d})$, where $\kappa\cong\bF_{p^d}$ denotes the residue field of $\mf{p}$.


To obtain quotients of type $\SL_3$, the element $\alpha^2$ of $\Gal(\bQ(t_{\zeta_n,k})/\bQ)\cong(\bZ/n)^\times$ (corresponding to $(x^2,1)$) has order $d$ in $\Gal(\bQ(t_{\zeta_n,k})$ and its image in $\Gal(\bQ(s_{\zeta_n,k})$ also has order $d$. Again by Chebotarev density, we obtain infinitely many primes with Frobenius element $\alpha^2$. For such primes, we have $\SU_{h,t_{\zeta_n,k}}(\kappa)\cong \SL_3(\bF_{p^d})$.
\end{proof}

\section{Proof of Theorem \ref{thm_main}}\label{section_tiep}
Recall that $R_0 := \bZ[q+q^{-1}]$ is the subring of $R := \bZ[q^{\pm1}]$ fixed by the involution sending $q\mapsto q^{-1}$. The Burau representation
$$\rho : B_4\lra U_h(R_0)\subset\GL_3(R)$$
sends $x,y\in F := \langle x,y\rangle\subset B_4$ to the matrices
$$X := \rho(x) = \thmatrix{-q}{1}{0}{0}{1}{0}{0}{1}{-q^{-1}},\qquad Y := \rho(y) = \thmatrix{1-q}{-q^{-1}}{q^{-1}}{1-q^2}{-q^{-1}}{0}{1}{-q^{-1}}{0}$$
These matrices have eigenvalues $1,-q,-q^{-1}$, and $Y = S_2XS_2^{-1}$, where
$$S_2 := \rho(\sigma_2) = \thmatrix{1}{0}{0}{q}{-q}{1}{0}{0}{1}$$
To prove Theorem \ref{thm_main}, we will show how to obtain, for all sufficiently large finite fields $k_0$, suitable specializations of the Burau representation which give $\SL_3(k_0)$ and $\SU_3(k_0)$ as characteristic quotients of $F$. In our formalism, such specializations should be understood as a base change. We recall the setup:

Let $r := q+q^{-1}\in R_0$. Thus $R$ is free of rank 2 over $R_0 = \bZ[r]$, generated by a root of the polynomial $f(T) := T^2 - rT + 1\in R_0[T]$. Let $k_0$ be a finite field of characteristic $p$. For $s\in k_0$, let
$$\SU_{h_s} := \SU_h\times_{R_0} k_0$$
be the base change of $\SU_h$ to the $R_0$-algebra $k_0$ with $R_0$-algebra structure given by the map sending $r\mapsto s$. Let $k$ be the quadratic $k_0$-algebra $k := R\otimes_{R_0}k_0 = k_0[T]/f(T)_{r:=s}$, with involution inherited from $R/R_0$. Let $t\in k$ be a fixed root of $f(T)_{r := s}$; this is guaranteed to be a unit, and the other root is $t^{-1}$. Thus $t,s$ satisfy the relation
$$s = t+t^{-1}$$
To prove Theorem \ref{thm_main}, we first need to understand for which values of $s$ do the finite groups $\SU_{h_s}(k_0)$ give groups of type $\SL_3(\bF_q)$ or $\SU_3(\bF_q)$.

\begin{lemma} Let $k_0$ be a finite field. Then
$$\SU_{h_s}(k_0)\cong \left\{\begin{array}{rl}
	\SL_3(k_0) & \text{if $T^2-sT+1$ splits into coprime factors over $k_0$} \\
	\SU_3(k_0) & \text{if $T^2-sT+1$ is irreducible over $k_0$ and $t^2+1\ne 0$}
\end{array}\right.$$
If $k_0$ has odd characteristic, $T^2-sT+1$ splits into coprime factors if and only if $s^2-4$ is a square in $k_0^\times$, and is irreducible if $s^2-4$ is a nonsquare. If $k_0$ has characteristic 2, $T^2-sT+1$ splits into coprime factors over $k_0$ if and only if $s \ne 0$ and $\Tr_{k_0/\bF_2}(s^{-1}) = 0$, and is irreducible if and only if $s\ne 0$ and $\Tr_{k_0/\bF_2}(s^{-1}) = 1$.
\end{lemma}

\begin{remark} 
If $T^2-sT+1$ is a square, $k\cong k_0[\epsilon]/\epsilon^2$, in which case by arguments similar to that used in Proposition \ref{prop_finite_groups}, we find that $\SU_{h_s}$ is nonreductive.	
\end{remark}

\begin{proof} If $f(T)_{r := s} = T^2-sT+1$ splits into distinct factors, then $k\cong k_0\times k_0$ and by Proposition \ref{prop_split_prime}, we have $\SU_{h_s}\cong\SL_{3,k_0}$. If $T^2-sT+1$ is irreducible, then $k = k_0(t)$ is a field, and $\SU_{h_s}$ is the special unitary group of the hermitian space $(k^3,h_s)$. From \eqref{eq_det_of_H}, we find that $h_s$ is nondegenerate if and only if $(t+1)(t^2+1)\ne 0$. Since $t$ generates the extension $k/k_0$, $t\ne -1$, so as long as $t^2+1\ne 0$, $h_s$ is nondegenerate, in which case $\SU_{h_s}(k_0)\cong \SU_3(k_0)$.

In odd characteristic, the discriminant of $T^2-sT+1$ is $s^2-4$, and hence it splits into coprime factors if and only if $s^2-4$ is a square in $k_0^\times$, and is irreducible if $s^2-4$ is a nonsquare. In characteristic 2, $T^2-sT+1$ is a square if and only if $s = 0$. If $s\ne 0$, then by the change of variable $X = T/s$, $T^2-sT+1$ is split over $k_0$ if and only if $g(X) := X^2 + X + s^{-2}$ is split, which happens if and only if $s^{-2}$ can be written as $a^2+a$ for some $a \in k_0$. The map $k_0\ra k_0$ sending $x\mapsto x^2+x$ is $\bF_2$-linear, and its image is precisely the kernel of $\Tr_{k_0/\bF_2} : k_0\ra \bF_2$. Thus we find that $T^2-sT+1$ splits into coprime factors over $k_0$ if and only $s \ne 0$ and $\Tr_{k_0/\bF_2}(s^{-1}) = 0$, and irreducible if $s\ne 0$ and $\Tr_{k_0/\bF_2}(s^{-1}) = 1$.
\end{proof}

Our main theorem \ref{thm_main} is implied by the following two results.

\begin{thm}\label{thm_tiep} Let $k_0 := \bF_{p^d}$ be a finite field such that $p^d\ge 9$. Then there exist elements $s,s'\in k_0$ satisfying $t^4 \ne 1$ where $s = t+t^{-1}$, and $(t')^4\ne 1$ where $s' = t'+t'^{-1}$, such that
\begin{enumerate}
	\item $T^2-sT+1$ splits into coprime factors over $k_0$ and $\SU_{h_s}(k_0)\cong \SL_3(k_0)$ is generated by the images of $X,Y$, and
	\item $T^2-s'T+1$ is irreducible over $k_0$ and $\SU_{h_{s'}}(k_0)\cong\SU_3(k_0)$ is generated by the images of $X,Y$.
\end{enumerate}
Moreover, (a) is also attainable for $k_0 = \bF_{8}$, and (b) is attainable for $k_0 = \bF_4$.
\end{thm}

The remaining cases of Theorem \ref{thm_main} were checked using GAP \cite{GAP4}.
\begin{prop} For a prime power $q$, let $Z_q$ denote a generator of the multiplicative group $\bF_q^\times$.
\begin{enumerate}
\item Let $t := Z_{64}^7$, and $s := t+t^{-1}\in\bF_8$. Then $\SU_{h_s}(\bF_8)\cong\SU_3(\bF_8)$ is generated by the images of $X,Y$.
\item Let $t := Z_{49}^6$, and $s := t+t^{-1}\in\bF_7$. Then $\SU_{h_s}(\bF_7)\cong\SU_3(\bF_7)$ is generated by the images of $X,Y$.
\item Let $t := Z_7^2$, and $s := t+t^{-1}\in\bF_7$. Then $\SU_{h_s}(\bF_7)\cong\SL_3(\bF_7)$ is generated by the images of $X,Y$.
\item Let $t := Z_{25}^8$, and $s := t+t^{-1}\in\bF_5$. Then $\SU_{h_s}(\bF_5)\cong\SU_3(\bF_5)$ is generated by the images of $X,Y$.
\item The groups $\PSL_3(\bF_2),\PSL_3(\bF_3),\PSL_3(\bF_4),\PSL_3(\bF_5),\PSU_3(\bF_2),\PSU_3(\bF_3)$ are not characteristic quotients of $F_2$.
\end{enumerate}
\end{prop}
\begin{proof} The specializations (a)-(d) were checked explicitly in GAP \cite{GAP4}. For each group $G$ in (e), we checked that for every surjection $\varphi : F_2\ra G$, the class of $\varphi\mod\Aut(G)$ has a proper stabilizer in $\Aut(F_2)$.
\end{proof}

\begin{proof}[Proof of Theorem \ref{thm_tiep}] For $a,d\in\bZ_{\ge 2}$, we say that a prime divisor $\ell\mid a^d-1$ is \emph{primitive} if $\ell$ does not divide $\prod_{i=1}^{d-1}(a^i-1)$. By a result of Zsigmondy \cite{Zsi92}, for any $a,d\in\bZ_{\ge 2}$, $a^d-1$ always admits a primitive prime divisor unless $(a,d) = (2,6)$ or $d = 2$ and $a+1$ is a power of 2. Any such prime divisor $\ell$ satisfies $\ell\ge d+1$. We will use this result repeatedly without mention.

We will also often use $X,Y$ to denote their images in $\SU_{h_s}(k_0)$ via the relevant specialization, and similarly we will use $v_1 = (1,q+1,q) ,v_2 = e_1, v_3 = e_3$ to denote the eigenbasis for $X$ defined in \S\ref{ss_the_representation}, with respective eigenvalues $1,-q,-q^{-1}$.

We begin with (a). Here we assume that
\begin{equation}\label{su10a}
 Q:= p^d \neq 2,3,5,8,
\end{equation}  

and set
$$\kappa:=\gcd(2,p-1).$$

\smallskip

(a1) We can find $t' \in \bF_{p^{2d}}^\times$ such that $-q=-t'$ has order $Q+1$. 
With this choice,
$$|X|=|Y|=Q+1 \geq 5.$$
The spectrum of $X$ in fact implies that
\begin{equation}\label{su10b}
 \mbox{ the smallest } j \geq 1 \mbox{ such that }X^j \mbox{ acts as a scalar is }Q+1.
\end{equation}

If $p=2$, then $t'=q=-q$ has order $Q+1 \geq 5$ by \eqref{su10a}. If $p > 2$ and $Q \equiv 1 \pmod{4}$, then 
$t'=q$ has order $(Q+1)/2 \geq 5$ by \eqref{su10a}. If $p > 2$ but $Q \equiv 3 \pmod{4}$, then 
$t=q'$ has order $Q+1 \geq 8$ by \eqref{su10a}. Thus in all cases we 
\begin{equation}\label{su11}
(t')^4 \neq 1,~(t')^{Q-1} \neq 1.
\end{equation}
We also have $q^{-1}=q^Q$, and so 
$$s'=t'+(t')^{-1}=q+q^{-1} = q + q^{p^d} \in \bF_{p^d}=k_0.$$
Furthermore, if $p > 2$ then $(s')^2-4 = (t'-(t')^{-1})^2$ is a non-square in $k_0$. Otherwise we would have $t' - (t')^{-1} \in k_0$, 
$s' = t'+(t')^{-1} \in k_0$, and so $q=t' \in k_0$. This would imply $(t')^{Q-1} =1$, contradicting \eqref{su11}. 
If $p=2$, then the polynomial $T^2-s'T+1$ has roots $t,t^{-1}$ which are not in $k_0$ and hence it does not split over $k_0$.

Note that, under the extra assumption
\begin{equation}\label{su10}
 d \geq 2 \mbox{ and }(p,d) \neq (2,3),
\end{equation}  
$p^{2d}-1$ admits a primitive prime divisor 
$$\ell \geq 2d+1 \geq 5,$$
which certainly divides $Q+1$.	

\smallskip
(a2) The preceding analysis in our paper shows that, with this choice $q=t'$, 
$$X, Y \in G:=\SU_{h_{s'}}(k_0) \cong \SU_3(k_0).$$ 
We can also verify this independently. Consider the basis $(v_1,e_1,e_3)$ of $k^3 = \bF_{p^{2d}}^3$ in which $X$ and $Y$ act via the matrices:
$$X':= \begin{pmatrix}1 & 0 & 0\\0 & -q & 0\\0 & 0 & -q^{-1} \end{pmatrix} \mbox{ and }
   Y':=\begin{pmatrix}1-q-q^{-1} & 1-q & 0\\ 0 & 0 & q^{-1}\\ q^2-q+1-q^{-1} & q^2-q+1 & 0 \end{pmatrix},$$    
respectively. Then $X'$ and $Y'$ preserve the Hermitian form $u \circ w$ on $k^3$ (semi-linear on $w$, i.e. 
$u \circ (\lambda w) = \lambda^{p^d}(u \circ w)$), with Gram matrix 
$$\begin{pmatrix} q+q^{-1} & 0 & 0\\0 & 1 & 0\\0 & 0 & 1 \end{pmatrix}$$
in this basis,  
which is $1/(q+2+q^{-1})$ times the form $h$ (represented by the matrix $D$ in \S\ref{ss_the_representation}). 

\smallskip
(a3) We will now show that 
$$\langle X,Y \rangle = G \cong \SU_3(k_0).$$ 
Assume the contrary:
$$\langle X,Y \rangle \leq M$$
for some maximal subgroup $M$ of $G$. We now apply the classification of maximal subgroups of $G \cong \SU_3(Q)$, as listed in 
\cite[Tables 8.5, 8.6]{BHR13}, to $M$. 

\smallskip
Note that, when $Q=p^d = p$, we have $p \geq 7$ by \eqref{su10a}, and so the order of $X$ in $G/Z(G)$ is $Q+1 \geq 8$
by \eqref{su10b}. This implies that $M$ does not belong to class $\cS$ in \cite[Table 8.6]{BHR13}.

\smallskip
(a4) By our choice, $q^4 \neq 0,1$. In particular, the eigenvalues $1$, $-q$, $-q^{-1}$ are pairwise distinct, and thus any line in $k^3$ that is 
$X$-invariant must be spanned by $e_1$, $v_1$, or $e_3$. (Here we have specialized $k \mapsto \bF_{p^{2d}}$.) 
Note that $Y(e_1) \notin ke_1$, $Y(e_3) \notin ke_3$. Next,
$$Y(v) = (1-q-q^{-1})e_1+(-q^2-q^{-1})e_2+(-q^{-1})e_3.$$
So, if $Y(v) \in kv$, we must have that $Y(v)=-q^{-2}v$ and hence $-q^2-q^{-1}= -q^{-1}-q^{-2}$, i.e. $q^4=1$, a contradiction. 
We have shown that there is no line in $k^3$ that is invariant under both $X$ and $Y$. In particular, $M$ fixes no line in $k^3$, and hence
cannot belong to any member of class $\cC_1$ in \cite[Table 8.5]{BHR13}. 

\smallskip
(a5) Suppose $M$ is a member of class $\cC_2$ in \cite[Table 8.5]{BHR13}, then 
$M \cong C_{Q+1}^2 \rtimes \fS_3$, the stabilizer of a decomposition of $k^3$ into an orthogonal sum of three anisotropic lines. 
Any element of $M$ that permutes these three lines cyclically have central order $3$ in $G$. By \eqref{su10b}, this is not the case for $X$.
So $X$, being an element in $M$, must fix at least one of these three lines, again a contradiction.

By assumption, $Q \neq 2,8$, which implies that $Q+1$ does not divide $3(Q^2-Q+1)$. Since $|X|=Q+1$, this shows that $M$ cannot belong to class 
$\cC_3$ in \cite[Table 8.5]{BHR13}. 

Suppose $M$ belongs to class $\cC_6$ in \cite[Table 8.5]{BHR13}. Then $M \leq 3^{1+2}_{+} \rtimes 2\mathsf{A}_4 < \SU_4(2)$. In particular, $X$, as an
element of $M$, must have order dividing $9$ or $12$, and hence $Q \in \{2,3,5,8,11\}$. Using \eqref{su10a}, we are left with $Q=11$, in which case
$X$ has order $12$. However, any element of order $12$ of $3^{1+2} \rtimes 2\mathsf{A}_4$ does not possess an eigenvalue $1$ in any 
$3$-dimensional irreducible representation over $\bC$ or $\overline{\bF}_{11}$ (as one can check using GAP \cite{GAP4}). 

Suppose $M$ is a member of type $\SU_3$ in class $\cC_5$ in \cite[Table 8.5]{BHR13}. In this case, $Q=Q_0^m$ for some prime $m \geq 3$, and 
so \eqref{su10} holds. In this case, $p^{2d}-1$ has a primitive prime divisor $\ell \geq 5$ which divides $|X|=Q+1$. Now the element
$X':=X^{(Q+1)/\ell}$ of order $\ell$ must belong to $\SU_3(Q_0)$. 
The primitivity of $\ell$ implies that $|X'|=\ell$ does not divide $Q_0^2-1= p^{2d/m}-1$. Hence
$\ell$ must divide $Q_0^3+1=p^{3d/m}+1$ (and so $m=3$). In such a case, any element of order $\ell$ in $\SU_3(Q_0)$ has spectrum of 
the form $\{\lambda,\lambda^{-Q_0},\lambda^{Q_0^2}\}$ for some $1 \neq \lambda \in \overline{\bF_p}^\times$. But this is a contradiction, since
$1$ is still an eigenvalue of $X'$.

\smallskip
(a6) Thus $p > 2$ and $M$ is a member of type $\SO_3$ in class $\cC_5$ in \cite[Table 8.5]{BHR13}:
$$M \cong C_a \times \SO_3(Q),$$
where $C_a = \langle z \rangle$ has order $a|\gcd(3,Q^2-1)$, and $z$ acts as a scalar $\alpha$ on $k^3$.
In this case, observe that
\begin{equation}\label{su12}
  X,Y \in \SO_3(Q).
\end{equation}  
Indeed, suppose for instance that $X = zX_1$ with $X_1 \in \SO_3(Q)$ but $\alpha \neq 1$. As $z$ is central of order coprime to $p$ and $|X|=Q+1$, we see
that $X_1$ is a semisimple element in $\SO_3(Q)$. As such, $X_1$ has spectrum $\{1,\beta,\beta^{-1}\}$ on $k^3$ for some 
$\beta \in \overline{\bF_p}^\times$. Thus 
$$\{1,-q,-q^{-1}\} = \{\alpha,\alpha\beta,\alpha\beta^{-1}\}.$$
But $\alpha \neq 1$, so $\alpha$ must be either $-q$ or $-q^{-1}$, and so $3=|\alpha|=Q+1$ by the choice of $-q=-t'$, a contradiction. 

\smallskip
The inclusion \eqref{su12} implies that $k^3$ supports a non-degenerate symmetric bilinear form $(\cdot,\cdot)$ which is invariant under both $X$ and $Y$.
Recall that $e_1$, $e_3$, and $v_1$ are eigenvectors of $X$ with eigenvalues $-q$, $-q^{-1}$, and $1$, respectively. As $q^2 \neq 1$, we must have
$$(e_1,e_1)=(e_3,e_3)=0.$$
Similarly, the eigenvector $S(e_1)=e_1+qe_2$ and $S(e_3)=e_2+e_3$ of $Y$ are isotropic, whence
$$(e_1,e_2)= -(q/2)(e_2,e_2),~(e_2,e_3)=-(1/2)(e_2,e_2).$$
We also have $v_1 \perp e_3$, which implies
$$(e_1,e_3) = -(q+1)(e_2,e_3)=((q+1)/2)(e_2,e_2).$$
Thus, denoting $(e_2,e_2)=2\epsilon$, we see that the Gram matrix of $(\cdot,\cdot)$ in the basis $(e_1,e_2,e_3)$ of $k^3$ is 
$$\Gamma:= \epsilon\begin{pmatrix} 0 & -q & q+1\\-q & 2 & -1\\q+1 & -1 & 0\end{pmatrix};$$
in particular, $\epsilon \neq 0$. 
Comparing the $(2,3)$-entries of $\Gamma$ and $\la{t}Y \Gamma Y$, we see that $-q^{-2}=-1$, i.e. $q^2=1$, again a contradiction.

We have therefore shown that $\langle X,Y \rangle = \SU_{h_{s'}}(k_0)$.

\smallskip
We now consider part (b). Here we assume that
\begin{equation}\label{sl10a}
 Q:= p^d \neq 2,3,4,5,7.
\end{equation}  

\smallskip
(b1) We can find $t \in \bF_{p^{2d}}^\times$ such that $-q=-t$ has order $Q-1$. 
With this choice,
$$|X|=|Y|=Q-1 \geq 7.$$
The spectrum of $X$ in fact implies that
\begin{equation}\label{sl10b}
 \mbox{ the smallest } j \geq 1 \mbox{ such that }X^j \mbox{ acts as a scalar is }Q-1.
\end{equation}
If $Q \equiv 1 \pmod{4}$ or $2|Q$, then 
$t=q$ has order $Q-1 \geq 7$ by \eqref{sl10a}. If $Q \equiv 3 \pmod{4}$, then 
$t=q$ has order $(Q-1)/2 \geq 5$ by \eqref{sl10a}. Thus in all cases we 
$$t^4 \neq 1.$$
Furthermore, 
$$s=t+t^{-1}=q+q^{-1} \in \bF_{p^d}=k_0,$$
and $s^2-4 = (t-t^{-1})^2$ is a square in $k_0$. If $p=2$, then the polynomial $T^2-sT+1$ has roots $t,t^{-1}$ which are in $k_0$ and hence it splits over $k_0$.

\smallskip
(b2) With the choice $q=t$ made above, clearly we have 
$$X, Y \in H:= \SL_3(k_0).$$
We will now show that 
$$\langle X,Y \rangle = H.$$ 
Assume the contrary:
$$\langle X,Y \rangle \leq M$$
for some maximal subgroup $M$ of $H$. We now apply the classification of maximal subgroups of $H \cong \SL_3(Q)$, as listed in 
\cite[Tables 8.3, 8.4]{BHR13}, to $M$. 

\smallskip
(b3) Note that the order of $X$ in $G/Z(G)$ is $Q-1 \geq 7$ by \eqref{sl10b}. This implies that $M$ does not belong to class $\cS$ in \cite[Table 8.4]{BHR13}.
(Indeed, the possibility $M = C_{\gcd(3,Q-1)} \times \PSL_2(7)$ can occur only when $Q = p \equiv 1,2,4 \pmod{7}$, which, together with \eqref{sl10a},
implies that $Q-1 \geq 10$.)

Next, by our choice, $q^4 \neq 0,1$. Hence the arguments in part (a4) show that there is no line in $k^3$ 
that is invariant under both $X$ and $Y$. Here we specialize $k \mapsto \bF_Q=\bF_{p^d}$. 
Suppose that $\langle X,Y \rangle$ fixes a plane $U \subset k^3$. Since $X$ acts on $U$ and its eigenvalues 
all belong to $k_0$, two of these eigenvalues occur in $U$. If one of these eigenvalues is $-q^{-1}$, then $e_3 \in U$. In this case,
$q^{-1}e_1 = Y(e_3) \in U$, so $e_1 \in U$. In turn, this implies that
$$(1-q^2)e_2=Y(e_1)-(1-q)e_1 - e_3 \in U,$$
so $U=k^3$, a contradiction. Thus $X|_U$ affords the eigenvalues $-q$ and $1$, and hence $e_1,v_1$ span $U$. But 
$$Y(e_1) =  (1-q)e_1+(1-q^2)e_2+e_3$$
is not a linear combination of $e_1$ and $v_1$, and so $Y(e_1) \notin  U$, again a contradiction.
We have shown that $M$ cannot fix any line or plane in $k^3$, and hence
cannot belong to members of class $\cC_1$ in \cite[Table 8.3]{BHR13}. 

Suppose $M$ is a member of class $\cC_2$ in \cite[Table 8.3]{BHR13}, then 
$M \cong C_{Q-1}^2 \rtimes \fS_3$, the stabilizer of a decomposition of $k^3$ into a direct sum of three lines. 
Any element of $M$ that permutes these three lines cyclically have central order $3$ in $G$. By \eqref{sl10b}, this is not the case for $X$.
So $X$, being an element in $M$, must fix at least one of these three lines, again a contradiction.

By assumption, $(Q-1) \nmid 9$, which implies that $Q-1$ does not divide $3(Q^2+Q+1)$. Since $|X|=Q-1$, this shows that $M$ cannot belong to class 
$\cC_3$ in \cite[Table 8.3]{BHR13}. 

Suppose $M$ belongs to class $\cC_6$ in \cite[Table 8.3]{BHR13}. Then $M \leq 3^{1+2}_{+} \rtimes 2\mathsf{A}_4 < \SU_4(2)$. In particular, $X$, as an
element of $M$, must have order dividing $9$ or $12$, and hence $Q \in \{2,3,4,5,7,13\}$. Using \eqref{sl10a}, we are left with $Q=13$, in which case
$X$ has order $12$. However, any element of order $12$ of $3^{1+2} \rtimes 2\mathsf{A}_4$ does not possess an eigenvalue $1$ in any 
$3$-dimensional irreducible representation over $\bC$ or $\overline{\bF}_{13}$.

\smallskip
(b4) Suppose $M$ is a member of type $\SL_3$ in class $\cC_5$ in \cite[Table 8.3]{BHR13}. Then $p^d=Q=Q_0^m$ for some prime $m$; 
in particular, $d \geq 2$. Also set
$$e:=\gcd(3,Q-1).$$
Then the element $X^e$ belongs to $\SL_3(Q_0)$, and admits eigenvalues $1$, and $(-q)^e$, 
$(-q^{-1})^e$ of order $(Q-1)/e$. Consider any maximal torus $T$ of $\SL_3(Q_0)$ that contains the semisimple element $X^e$. 
If $T \cong C_{Q_0-1}^2$, then $(Q-1)/e$ divides 
$$Q_0-1 = p^{d/m}-1 \leq p^{d/2}-1 < \frac{p^d-1}{3} \leq \frac{Q-1}{e},$$
a contradiction. Suppose $T \cong C_{Q_0^2+Q_0+1}$. Then any element in $T$ has spectrum of 
the form $\{\lambda,\lambda^{Q_0},\lambda^{Q_0^2}\}$ for some $1 \neq \lambda \in \overline{\bF_p}^\times$. As $1$ is an eigenvalue of $X^e \in T$, 
we see that the corresponding $\lambda$ for $X^e$ must be $1$, and so $X^e=\mathrm{Id}$, again a contradiction.
In the remaining cases, $T \cong C_{Q_0^2-1}$, which consists of elements with 
spectrum of the form $\{\mu,\mu^{Q_0},\mu^{-Q_0-1}\}$ for some 
$\mu \in \overline{\bF_p}^\times$. As $1$ is an eigenvalue of multiplicity one for $X^e$, we get $\mu^{Q_0+1}=1$ for 
the corresponding $\mu$ for $X^e$. If $Q \neq 16$, then all eigenvalues of $X$ must be of order dividing 
$$e(Q_0+1) \leq e(p^{d/2}+1) < p^d-1,$$
which is a contradiction since the eigenvalue $-q$ of $X$ has order $Q-1$. If $Q = 16$, then $M = \SL_3(4)$, and one can check that no element
of order $Q-1=15$ can have the same spectrum as that of $X$, again a contradiction.

\smallskip
(b5) Suppose $M$ is a member of type $\SU_3$ in class $\cC_8$ in \cite[Table 8.3]{BHR13}. Then $p^d=Q=Q_0^2$ 
and $M \cong \SU_3(Q_0) \times C_f$,
where 
$$f:=\gcd(3,Q_0-1).$$
By \eqref{sl10a} we have $Q_0 \neq 2$. Assume in addition that $Q_0 \neq 4$. 
Now the element $X^f$ belongs to $\SU_3(Q_0)$, and admits eigenvalues $1$, and $(-q)^f$, 
$(-q^{-1})^f$ of order $(Q-1)/f$. Consider any maximal torus $T$ of $\SU_3(Q_0)$ that contains the semisimple element $X^f$. 
If $T \cong C_{Q_0+1}^2$, then $(Q-1)/f=(Q_0^2-1)/f$ divides $Q_0+1$, and so $Q_0 = 2$ or $4$, a contradiction.
Suppose $T \cong C_{Q_0^2-Q_0+1}$. Then any element in $T$ has spectrum of 
the form $\{\lambda,\lambda^{-Q_0},\lambda^{Q_0^2}\}$ for some $1 \neq \lambda \in \overline{\bF_p}^\times$. As $1$ is an eigenvalue of $X^f \in T$, 
we see that the corresponding $\lambda$ for $X^f$ must be $1$, and so $X^f=\mathrm{Id}$, again a contradiction.
In the remaining cases, $T \cong C_{Q_0^2-1}$, which consists of elements with 
spectrum of the form $\{\mu,\mu^{-Q_0},\mu^{Q_0-1}\}$ for some 
$\mu \in \overline{\bF_p}^\times$. As $1$ is an eigenvalue of multiplicity one for $X^f$, we get $\mu^{Q_0-1}=1$ for 
the corresponding $\mu$ for $X^f$. But then all eigenvalues of $X$ must be of order dividing 
$$f(Q_0-1) \leq 3(Q_0-1) < Q_0^2-1,$$
a contradiction since the eigenvalue $-q$ of $X$ has order $Q-1$. 

Now we consider the case $Q_0=4$, so that $Q=16$ and 
$$M = \langle t \rangle \times \SU_3(4),$$
where $t$ acts as a scalar $\beta$ of order $3$ on $k^3$. Then we can write $X=t^iX_1$, where $X_1$ is a semisimple element in $\SU_3(4)$ and 
$i \in \bZ$. Consider any maximal torus $T$ of $\SU_3(4)$ that contains $X_1$. 
If $T \cong C_5^2$, then $X_1$ is conjugate to $\diag(\gamma_1,\gamma_2,\gamma_3)$ where $\gamma_j^5=1$. Since $1$ is an eigenvalue 
of $X$, we may assume that $\beta^i\gamma_1=1$, which implies that $\beta^i=\gamma_1=1$ and $X=X_1$. But in this case no eigenvalue of 
$X$ can have order $15$, a contradiction. 
Suppose $T \cong C_{13}$. Then any element in $T$ has spectrum of 
the form $\{\lambda,\lambda^{-Q_0},\lambda^{Q_0^2}\}$ for some $1 \neq \lambda \in \overline{\bF_2}^\times$ and 
$\lambda^{13}=1$. As $1$ is an eigenvalue of $X=t^iX_1$, 
we see that the corresponding $\lambda$ for $X_1$ must be $1$ and $\beta^i=1$, and so $X=\mathrm{Id}$, again a contradiction.
In the remaining cases, $T \cong C_{15}$, which consists of elements with 
spectrum of the form $\{\mu,\mu^{-4},\mu^{3}\}$ for some 
$\mu \in \overline{\bF_2}^\times$ with $\mu^{15}=1$. As $1$ is an eigenvalue of multiplicity one for $X=t^iX_1$, for 
the corresponding $\mu$ for $X_1$ we have 
$$\beta^i\mu=1 \mbox{ or }\beta^i\mu^4=1 \mbox{ or }\beta^i\mu^3 = 1.$$
Since $|\beta|=3$, we get $\mu^9= 1$, and hence $\mu^3=1$. 
But then all eigenvalues of $X$ must be of order dividing $3$, 
a contradiction since the eigenvalue $-q$ of $X$ has order $Q-1=15$. 

\smallskip
(b6) The only remaining possibility is that $p > 2$ and $M$ is a member of type $\SO_3$ in class $\cC_8$ in \cite[Table 8.3]{BHR13}: 
$$M \cong \SO_3(Q) \times C_{\gcd(3,Q-1)}.$$ 
This is however ruled out by the same arguments as in part (a6). 
Hence we have shown that $\langle X,Y \rangle = \SL_3(k_0)$.
\end{proof}

\section{$\PSL_2(\bF_q)$ is not a characteristic quotient of $F_2$}\label{section_markoff}
In this section we show that for any prime power $q$, $\PSL_2(\bF_q)$ is not a characteristic quotient of $F_2$. In particular, this shows that the groups $\PSL_3(\bF_q),\PSU_3(\bF_q)$ that we find are the lowest (untwisted Lie) rank finite simple groups of Lie type that are characteristic quotients of $F_2$. The meat of the proof is a short explicit calculation, though first we will recall the relevant formalism. For more details, see \cite[\S5.2]{Chen21}.

Let $\SL_2$ denote the affine group scheme over $\bZ$ with ring $A = \bZ[a,b,c,d]/(ad-bc-1)$. Representations of $F_2$ in $\SL_2$ form a scheme $\Hom(F_2,\SL_2) \cong \SL_2\times\SL_2\cong \Spec A\otimes_\bZ A$ whose $k$-valued points, for any ring $k$, are exactly the representations of $F_2$ in $\SL_2(k)$. Letting superscripts denote fixed points, the quotient
$$X_{\SL_2} := \Hom(F_2,\SL_2)\git\GL_2 = \Spec (A\otimes_\bZ A)^{\GL_2}$$
by the conjugation action of $\GL_2$ (equivalently, the diagonal action by conjugation on $\SL_2\times\SL_2$) is the (integral) \emph{character variety} for $\SL_2$-representations of $F_2$. The natural action of $\Aut(F_2)$ on $\Hom(F_2,\SL_2)$ induces an action of $\Out(F_2)$ on the character variety $X_{\SL_2}$. Much of the theory of $\SL_2$-representations of $F_2$ follows from the following fundamental result.

\begin{thm}[Brumfiel-Hilden] The map $\Tr : X_{\SL_2}\ra\bA^3_\bZ$ sending the class of a representation $\varphi : F_2\ra\SL_2(k)$ (for any ring $k$) to the point $(\tr\varphi(a),\tr\varphi(b),\tr\varphi(ab))\in k^3$ is an \emph{isomorphism}.
\end{thm}

\begin{proof} The statement over $\bC$ was known to Fricke and Vogt, see \cite{Gold04,Gold09}. The statement over $\bZ$ is due to Brumfiel-Hilden \cite[Prop 3.5 and 9.1(ii)]{BH95}, also see \cite[Theorem 5.2.1]{Chen21}.	
\end{proof}

Beware that while $\GL_2$-equivalent representations give rise to the same point in the character variety, in general the preimage of a point in $X_{\SL_2}$ is the union of multiple $\GL_2$-orbits in $\Hom(F_2,\SL_2)$. These difficulties disappear if we consider points valued over finite fields corresponding to absolutely irreducible representations:

\begin{thm} Let $Z\subset\bA^3_\bZ$ denote the hypersurface given by $x^2+y^2+z^2-xyz = 4$. For any finite field $\bF_q$, the map $\Tr : X_{\SL_2}\rightiso\bA^3_\bZ$ induces a \emph{bijection}
$$\{\text{Absolutely irreducible representations }\varphi : F_2\ra\SL_2(\bF_q)\}/\GL_2(\bF_q)\;\;\rightiso\;\; \bF_q^3 \setminus Z(\bF_q)$$
\end{thm}
\begin{proof} See \cite[Theorem 5.2.10]{Chen21}.	
\end{proof}

Since the conjugation action of $\GL_2(\bF_q)$ on $\SL_2(\bF_q)$ induces the full action of $\Aut(\SL_2(\bF_q))$ and the action of $\Out(F_2)$ preserves (absolute) irreducibility, we find that characteristic quotients $F_2\twoheadrightarrow\SL_2(\bF_q)$ correspond to $\Out(F_2)$-fixed points on $\bF_q^3 \setminus Z(\bF_q)$.

Since two maps to $\SL_2(\bF_q)$ give rise to the same map to $\PSL_2(\bF_q)$ if and only if the images of $a,b$ under the two maps differ by $\pm I$, it follows that $\Epi(F_2,\PSL_2(\bF_q))/\Aut(\PSL_2(\bF_q))$ can be identified with a subset of the orbit space $(\bF_q^3\setminus Z(\bF_q))/V$, where $V\cong\bZ/2\times\bZ/2$ acts by negating two of the coordinates (note that $V$ preserves $Z$). Accordingly, characteristic quotients $F_2\twoheadrightarrow\PSL_2(\bF_q)$ correspond to $\Out(F_2)$-fixed points on $(\bF_q^3\setminus Z(\bF_q))/V$.

\begin{thm}\label{thm_fixed_points} Let $k$ be a ring. The action of $\Out(F_2)$ on $k^3\cong X_{\SL_2}(k)$ descends to an action on the orbit space $k^3/V$, where $V\cong\bZ/2\times\bZ/2$ acts by negating two coordinates. The only points in $k^3/V$ fixed by $\Out(F_2)$ are the $V$-orbits of $(0,0,0)$ and $(2,2,2)$.
\end{thm}
\begin{proof} The group $\Aut(F_2)$ is generated by three elements $r,s,t$. The definitions of $r,s,t$ and the corresponding automorphisms of $\bA^3_\bZ\cong X_{\SL_2}$ are given as follows (see \cite[Lemma 5.2.4]{Chen21})
$$\begin{array}{rcl}
r : (a,b) & \mapsto & (a^{-1},b) \\
s : (a,b) & \mapsto & (b,a) \\
t : (a,b) & \mapsto & (a^{-1},ab)	
\end{array} \quad\stackrel{\Tr_*}{\lra}\quad
\begin{array}{rcl}
\Tr_*(r) : (x,y,z) & \mapsto & (x,y,xy-z) \\
\Tr_*(s) : (x,y,z) & \mapsto & (y,x,z)	\\
\Tr_*(t) : (x,y,z) & \mapsto & (x,z,y)
\end{array}$$
From this, we find that $s,t$-invariance implies that $\pm X = \pm Y = \pm Z$. Assuming this, $r$-invariance additionally implies that $X \in\{0,-2,2\}$. One checks that the $V$-orbit of $(-2,2,2)$ is not $r$-invariant, and hence the only fixed points are the orbits $V\cdot(0,0,0)$ and $V\cdot(2,2,2)$.
\end{proof}

In the proof above, beware that while $\Aut(F_2)$ acts on the \emph{right} on $X_{\SL_2}$, the equations for $\Tr_*(r),\Tr_*(s),\Tr_*(t)$ describe a \emph{left} action on $\bA^3$. Thus $\Tr_*$ is an \emph{anti}-homomorphism $\Aut(F_2)\ra\Aut(\bA^3)$.

\begin{cor}\label{cor_SL2} For any prime power $q$, neither $\SL_2(\bF_q)$  nor $\PSL_2(\bF_q)$ is a characteristic quotient of $F_2$.	
\end{cor}
\begin{proof} It suffices to show that $\PSL_2(\bF_q)$ is not a characteristic quotient. For this, by Theorem \ref{thm_fixed_points}, it suffices to note that the $V$-orbits of $(0,0,0)$ and $(2,2,2)$ cannot correspond to \emph{surjective} representations $\varphi : F_2\ra\PSL_2(\bF_q)$. Since $(2,2,2)$ satisfies $x^2+y^2+z^2-xyz=4$, the corresponding $\SL_2$-representation is not absolutely irreducible, and hence cannot be surjective. This also handles $(0,0,0)$ for even $q$, so it remains to consider $(0,0,0)$ for $q$ odd.

Note that a trace 0 matrix $A\in\SL_2(\bF_q)$ has characteristic polynomial $T^2 + 1$. Thus $A$ has order 4 and $A^2 = \spmatrix{-1}{0}{0}{-1}$, so that $A$ has order 2 in $\PSL_2(\bF_q)$. It follows that for $q$ odd, the point $(0,0,0)$ corresponds to a representation whose image in $\PSL_2(\bF_q)$ is isomorphic to $\bZ/2\times\bZ/2$, and hence is not surjective.
\end{proof}

\section{Relation with noncongruence modular curves}\label{section_noncongruence}
In this section we explain how our main theorem disproves two conjectures in the theory of noncongruence subgroups of $\SL_2(\bZ)$. Recall that a subgroup $\Gamma\le\SL_2(\bZ)$ is \emph{congruence} if it contains $\Gamma(n) := \ker(\SL_2(\bZ)\ra\SL_2(\bZ/n))$ for some integer $n\ge 1$. Otherwise, it is \emph{noncongruence}. Let us define the \emph{congruence degree} of a subgroup $\Gamma\le\SL_2(\bZ)$ to be the index inside $\SL_2(\bZ)$ of the \emph{congruence closure} $\Gamma^c$ of $\Gamma$; this $\Gamma^c$ is the intersection of all congruence subgroups containing $\Gamma$. Noncongruence subgroups of congruence degree 1 are \emph{totally noncongruence} \cite[\S4.4]{Chen18}. Since $\SL_2(\bZ)$ admits the free product $\PSL_2(\bZ)\cong\bZ/2 *\bZ/3$ as a quotient, it follows that $\SL_2(\bZ)$ contains many finite index noncongruence subgroups.\footnote{If $G$ is a finite group which has a nonabelian simple composition factor not of type $\PSL_2(\bF_p)$, then the kernel of any surjection $\SL_2(\bZ)\ra G$ is noncongruence.}

In the theory of modular curves, many congruence subgroups can be understood as the $\SL_2(\bZ)$-stabilizer of elements of $\Epi(\bZ^2,G)$, where $G$ is a finite quotient of $\bZ^2$. Identifying $\bZ^2$ as the fundamental group of an elliptic curve, this leads to a moduli interpretation of the modular curve $\cH/\Gamma$, where $\cH$ denotes the complex upper half plane, and $\Gamma\le\SL_2(\bZ)$ is a congruence subgroup acting on $\cH$ by mobius transformations. For example, if $\varphi : \bZ^2\ra\bZ/n$ is a surjection with stabilizer $\Gamma_\varphi := \Stab_{\SL_2(\bZ)}(\varphi)$, then $\cH/\Gamma$ is the moduli space of elliptic curves $E$ equipped with a point of order $n$ (equivalently, a $\bZ/n$-cover $E'\ra E$), see \cite[Prop 2.2.12]{Chen18} or \cite[\S1.5]{DS06}. More generally, congruence modular curves are moduli spaces of elliptic curves with \emph{abelian covers}. Such abelian covers can be equivalently described by torsion data.

By \cite[Cor N4]{MKS04}, the abelianization map $F_2\ra\bZ^2$ induces a rather exceptional isomorphism $\Out(F_2)\cong\GL_2(\bZ)$, restricting to an isomorphism $\Out^+(F_2)\cong\SL_2(\bZ)$. It follows that the action of $\SL_2(\bZ)$ on $\Epi(\bZ^2,G)$ is a special case of the more general action
$$\SL_2(\bZ)\cong\Out^+(F_2)\circlearrowright\Epi(F_2,G)/\Inn(G)$$
If $\varphi : F_2\ra G$ is a surjection, let $\Gamma_\varphi$ denote the $\SL_2(\bZ)\cong\Out^+(F_2)$-stabilizer of its conjugacy class $\varphi\mod\Inn(G)$. Viewing $F_2$ as the fundamental group of an elliptic curve with the origin removed, it was shown in \cite{Chen18} that $\cH/\Gamma_\varphi$ is a moduli space of elliptic curves with branched $G$-Galois covers, only ramified above the origin. By a result of Asada, every finite index subgroup of $\SL_2(\bZ)$ contains $\Gamma_\varphi$ for some appropriate finite group $G$ and surjection $\varphi : F_2\ra G$. This endows every noncongruence modular curve with a moduli interpretation. In some sense this puts noncongruence modular curves on a similar footing to congruence modular curves, and opens the door to a deeper investigation of the noncongruence side, guided by the wildly successful congruence theory.

A fundamental question that arises in this theory is: When does a surjection $\varphi : F_2\ra G$ have a congruence or noncongruence stabilizer? In the former case we say that $\varphi$ is congruence, and in the latter that $\varphi$ is noncongruence. It's easy to see that if $G$ is abelian, then every surjection is congruence. In \cite{CD17}, it is also shown that the same is true if $G$ is metabelian. In \cite[Theorem 4.4.10]{Chen18}, it was shown that if $x,y$ are generators of $F_2$, then if $G$ is not the trivial group and the orders of $\varphi(x),\varphi(y),\varphi(xy)$ are pairwise coprime, then $\Gamma_\varphi$ is noncongruence, and its congruence closure\footnote{The congruence closure of finite index a subgroup $\Gamma$ of $\SL_2(\bZ)$ is the intersection of all congruence subgroups containing $\Gamma$.} is $\SL_2(\bZ)$. These results support the following philosophy:
\begin{equation}\label{eq_philosophy}
\parbox{\linewidth-5em}{Surjections onto highly nonabelian groups $G$ should tend to have highly noncongruence stabilizers.}
\end{equation}
In \cite{Chen18}, the author attempted to make this philosophy more precise via a pair of conjectures. First, \cite[Conjecture 4.4.10]{Chen18} asserts that for a fixed finite group $G$, either every surjection onto $G$ is congruence, or they are all noncongruence. Second, \cite[Conjecture 4.4.1]{Chen18} asserts that for nonsolvable $G$, all surjections onto $G$ should be noncongruence. While supported by computational data at the time\footnote{At the time, all stabilizers of surjections onto groups of order $\le 255$ and simple groups of order $\le |\Sz(8)| = 29120$ had been computed. See \cite[Appendix B]{Chen18}.}, both conjectures are false.

At the time of writing of this paper, all stabilizers of surjections onto nonabelian finite simple groups of order $\le |\J_1| = 175560$ ($\J_1$ is the Janko group) have been computed. Amongst these, there are only two $\Aut^+(F_2)\times\Aut(G)$-orbits of congruence surjections, both coming from characteristic quotients obtained via the Burau representation. These give rise to four $\Aut^+(F_2)$-orbits of congruence surjections onto $\PSU_3(\bF_4)$ (the remaining 29 being noncongruence), and two $\Aut^+(F_2)$-orbits of congruence surjections onto $\PSU_3(\bF_5)$ (the remaining 67 being noncongruence). Amongst simple groups, these are the smallest counterexamples to \cite[Conjectures 4.4.1, 4.4.10]{Chen18}. Theorem \ref{thm_main} in fact produces infinitely many counterexamples to \cite[Conjecture 4.4.1]{Chen18}: For primes $p$, groups $G$ of type $\PSL_3(\bF_p)$ or $\PSU_3(\bF_p)$ have $\Out(G)\le 6$, and hence characteristic surjections onto such groups have $\SL_2(\bZ)$-stabilizers of index $\le 6$. Since all subgroups of index $\le 6$ are congruence, the characteristic quotients onto such groups produced by Theorem \ref{thm_main} are all congruence.

Other than the 6 conjugacy classes of characteristic quotients onto $\PSU_3(\bF_4),\PSU_3(\bF_5)$, the remaining 1328 $\Aut^+(F_2)$-orbits of surjections onto simple groups of order $\le |\J_1| = 175560$ are all noncongruence with congruence degree $\le 3$. This leads to the natural question:

\begin{question} Is there a \emph{global} upper bound for the congruence degrees of stabilizers of surjections onto nonabelian finite simple groups?
\end{question}

We have also checked that amongst the 60 groups $G$ of order $\le 255$ with solvable length $\ge 4$, all surjections onto $G$ are noncongruence \cite[Appendix B.1]{Chen18}.

\begin{question} Is there a lower bound on solvable length that can guarantee either the existence of noncongruence surjections, or the nonexistence of congruence surjections?
\end{question}

In light of our results, we propose the following weakening of \cite[Conjecture 4.4.1]{Chen18}:
\begin{conj} If $G$ is a nonabelian finite simple group, then there exists a surjection $\varphi : F_2\ra G$ such that $\Gamma_\varphi$ is noncongruence. In other words, the kernel of the $\SL_2(\bZ)$ action on $\Epi(F_2,G)/\Inn(G)$ is noncongruence.	
\end{conj}

This conjecture, which we will address in a sequel paper, still does not capture the empirical rareness of congruence surjections onto simple groups $G$. Indeed, it's unclear how to formulate a statement to this end that might be true. The fact that the line between congruence and noncongruence is so difficult to pin down in this setting contrasts with the clean characterizations of noncongruence subgroups given by the Hecke theory \cite{Berg94} and the unbounded denominators property for noncongruence modular forms \cite{CDT21}. It would be very interesting to see if a moduli interpretation for Hecke operators on noncongruence modular curves could shed light on this problem.

Geometrically, our main theorem \ref{thm_main} implies a strange correspondence between abelian covers (or torsion points) on an elliptic curve $(E,O)$ and $\PSL_3(\bF_p)$ (or $\PSU_3(\bF_p)$)-covers of $E$ ramified above the origin $O$. Our methods imply a relation between such examples and Zariski dense representations of $B_4$ whose restrictions to $F$ have large image. We end by noting that an elliptic curve $(E,O)$ simultaneously determines an abelian variety, as well as a hyperbolic curve $E^\circ := E \setminus O$. Hyperbolic curves are also examples of \emph{anabelian varieties}, in the sense of Grothendieck's anabelian geometry \cite{Gro97, MNT98}, which are distinguished by the property that their structure as curves over a number field $K$ is completely determined by the action of $\Gal(\Qbar/K)$ on their profinite fundamental group. The results of \cite{Asa01} and \cite{Chen18} suggest that while congruence modular curves capture the essence of elliptic curves as abelian varieties, noncongruence modular curves capture their essence as anabelian curves. The problem of making precise the philosophy \eqref{eq_philosophy} amounts to mapping out the boundary where the congruence theory stops and the noncongruence theory begins.

\section{Appendix}

\subsection{Algebraic groups embedded in $\GL_n$}
Let $R$ be a Dedekind domain with fraction field $K$. If $G$ is an algebraic group over $K$ equipped with a (homomorphic) embedding $G\hookrightarrow \GL_{n,K}$ then we may define $G(R) := G(K)\cap\GL_n(R)$ as the set of $R$-points of $\GL_{n,R}$ whose intersection (as a closed subscheme) with the generic fiber $\GL_{n,K}$ lies in $G$. Sometimes it is nice to be able to consider $G(R)$ as the actual $R$-points of an algebraic group scheme over $R$. Indeed this can be done by taking schematic closures:

\begin{prop}\label{prop_schematic_closure} Let $R$ be a Dedekind domain with fraction field $K$. Let $H$ be a reduced flat affine group scheme over $R$, and let $G$ be a reduced affine $K$-group scheme equipped with a homomorphic closed immersion $G\hookrightarrow H_K$ (for example one might want to take $H = \GL_{n,R}$). Let $\ol{G}$ denote the schematic closure of $G$ inside $H$, defined as the scheme theoretic image \cite[01R5]{stacks} of the composition $G\hookrightarrow H_K\hookrightarrow H$. Then $\ol{G}$ is an $R$-flat closed subscheme of $H$ with a unique structure of an $R$-group scheme making the morphisms $G\ra\ol{G}\ra H$ homomorphisms of $R$-group schemes. Moreover $\ol{G}_K := \ol{G}\times_{\Spec R}\Spec K$ is canonically isomorphic to $G$, and we have $G(K)\cap H(R) = \ol{G}(R)$.
\end{prop}
Note that since $G,H$ are affine, the map $G\hookrightarrow H$ is affine and hence quasi-compact \cite[01S8]{stacks}. Since $G$ is reduced, the scheme theoretic image is the closure of the image with the reduced induced structure \cite[056B]{stacks}\footnote{I think the referenced lemma also requires the quasi-compactness of ``$f$'', even if it is not stated.}. Thus the schematic closure is just the topological closure with the reduced induced structure.
\begin{proof} The statement (in a somewhat weaker form) is \cite[\S2]{Con14}, and the proof is due to Damian R\"{o}ssler. Let $H = \Spec A$, and $G = \Spec B$, so we have by assumption a surjection $f : A_K := A\otimes_R K\lra B$ coming from $G\hookrightarrow H_K$ and an injection $g : A\hookrightarrow A_K$ coming from $H_K\hookrightarrow H$. Let $\mu := f\circ g : A\ra B$, then by definition we have $\ol{G} = \Spec\mu(A)$. Since $\mu(A)$ is a subring of the $R$-flat module $B$, it is torsion-free, hence flat since $R$ is Dedekind. It remains to show that $\mu(A)$ admits the structure of a Hopf algebra compatible with those of $A,B$. Let $m_A,m_B$ be the comultiplications of $A,B$. Then viewing $B$ as a group scheme over $R$, we are looking to fill in the dotted arrow in the diagram
\[\begin{tikzcd}
	A\ar[r,"\mu"]\ar[d,"m_A"] & \mu(A)\ar[r,hookrightarrow]\ar[d,dashrightarrow] & B\ar[d,"m_B"] \\
	A\otimes_R A\ar[r,"\mu\otimes\mu"] & \mu(A)\otimes_R\mu(A)\ar[r,"h"] & B\otimes_R B
\end{tikzcd}\]
Indeed, any such arrow would result in a Hopf algebra structure making the maps $G = \Spec B\ra \Spec\mu(A)\ra\Spec A = H$ homomorphisms of $R$-group schemes (use the left half of the diagram, noting that $A\ra\mu(A)$ is surjective). For such a dashed arrow to exist, it clearly suffices to check that the map $h$ is injective. Set $M := \mu(A)\otimes_R\mu(A)$. Note that $\mu(A)\otimes_R K = B$, and since $K\otimes_R K = K$, $M\otimes_R K = B\otimes_R B$, so the map $h$ can be identified with the canonical map $M\ra M\otimes_R K$. Identifying $M\otimes_R K$ with the localization $(R - \{0\})^{-1}M$, we find that $\ker(M\ra M\otimes_R K)$ is the $R$-torsion submodule of $M$, and so the map $h$ is injective if and only if $\mu(A)\otimes_R\mu(A)$ is torsion-free, which holds for us since $\mu(A)$ is $R$-flat. In this case it's clear that the dashed arrow is even unique, so the group structure on $\ol{G}$ is uniquely determined by $G\hookrightarrow H$.

Finally, let $\eta : \Spec K\ra\Spec R$ be induced by the inclusion $R\subset K$. Given $g\in\ol{G}(R)$ with $g : \Spec R\ra\ol{G}$, the composition $\Spec R\stackrel{g}{\ra}\ol{G}\ra H$ gives an $R$-point of $H$, and the composition $\Spec K\stackrel{\eta}{\ra}\Spec R\stackrel{g}{\ra}\ol{G}$ gives a $K$-point of $\ol{G}$, and hence a $K$-point of $G$. The latter construction defines a map $\ol{G}(R)\ra G(K)$ which is injective because $\ol{G}$ is separated (since it is affine). Thus we may view $\ol{G}(R)$ and $G(K)\cap H(R)$ as subgroups of $G(K)$, with $\ol{G}(R)\subset G(K)\cap H(R)$. To show they are equal, suppose we have an $R$-point $x : \Spec R\lra H$ such that $x\circ\eta\in G(K)$. Then $x$ defines an irreducible closed subscheme of $H$ inside which the image of $x\circ\eta$ (a point of $G$) is dense. Since $\ol{G}$ is the closure of $G$, $\ol{G}$ contains the image of $x$. Since everything in sight is reduced, we conclude that $x$ factors through $\ol{G}$ \cite[0356]{stacks}; thus $\ol{G}(R) = G(K)\cap H(R)$ as desired.
\end{proof}

\begin{remark}\label{remark_equational_characterization} If $G\hookrightarrow H_K\hookrightarrow H$ is given on rings by $\mu : A\ra A_K\ra B$, then the proof of the proposition shows that the schematic closure $\ol{G} = \Spec\mu(A)$. If there exists an ideal $I\subset A$ such that $A_K\ra B$ is given by $A_K\ra A_K/IA_K$, then we would have $\mu(A) = A/I$ (Note that $IA_K = I\otimes_R K$, and use exactness of $\otimes_R K$). In practice, this means that if $G$ is cut out inside $H_K$ by equations ``with coefficients in $R$'', then $\ol{G}$ is cut out inside $H$ by the same equations. In particular, the scheme cut out by those equations is flat and finite type over $R$.
\end{remark}

\bibliography{CLNT_references}

\gap

\begin{tabular}{l}
William Y. Chen \\
Department of Mathematics, Rutgers University, Hill Center for the Mathematical Sciences \\
110 Frelinghuysen Rd, Piscataway, NJ, 08854-8019, USA \\
Email address: \url{oxeimon@gmail.com}
\end{tabular}

\begin{tabular}{l}
Alexander Lubotzky \\
Faculty of Mathematics and Computer Science, The Weizmann Institute of Science \\
234 Herzl Street, Rehovot 7610001, Israel \\
Email address: \url{alex.lubotzky@mail.huji.ac.il}
\end{tabular}

\begin{tabular}{l}
Pham Huu Tiep \\
Department of Mathematics, Rutgers University, Hill Center for the Mathematical Sciences \\
110 Frelinghuysen Rd, Piscataway, NJ, 08854-8019, USA \\
Email address: \url{pht19@math.rutgers.edu}
\end{tabular}

\end{document}